\documentclass[10pt]{amsart}

\usepackage{amsmath}
\usepackage{array,subfigure}
\usepackage{graphicx}

\newtheorem{thm}{Theorem}[section]
\newtheorem{prop}[thm]{Proposition}
\newtheorem{cor}[thm]{Corollary}
\newtheorem{lem}[thm]{Lemma}
\newtheorem{defn}[thm]{Definition}

\newenvironment{remark}{\refstepcounter{thm} \medskip \noindent {\bf  Remark \arabic{section}.\arabic{thm}}}{\hfill\mbox{}\bigskip}

\newcounter{num}

\newenvironment{thmlist}{\begin{list}{(\roman{num})}{\usecounter{num}\setlength{\leftmargin}{25pt}
\setlength{\itemindent}{0pt}\setlength{\labelwidth}{20pt}\setlength{\labelsep}{5pt}\setlength{\itemsep}{0in}}}{\end{list}}

\newcommand{\C}{\mathbb{C}}

\newcommand{\R}{\mathbb{R}}
\newcommand{\Z}{\mathbb{Z}}

\newcommand{\cps}{\mathbb{C}P}

\newcommand{\ol}[1]{\bar{#1}}

\newcommand{\AC}{\operatorname{AC}}
\newcommand{\Card}{\operatorname{Card}}

\newcommand{\contr}{\,\lrcorner\,}

\newcommand{\Hom}{\operatorname{Hom}}
\newcommand{\im}{\operatorname{Im}}

\newcommand{\Ric}{\operatorname{Ric}}

\newcommand{\Vol}{\operatorname{Vol}}
\newcommand{\inj}{\operatorname{inj}}

\newcommand{\Spec}{\operatorname{Spec}}
\newcommand{\SF}{\operatorname{SF}}

\newcommand{\inter}{\operatorname{Int}}

\title{Regularity of asymptotically conical Ricci-flat K\"{a}hler metrics}
\author{Craig van Coevering}
\address{Department of Mathematics, Massachusetts Institute of Technology, 77 Massachusetts Avenue, Cambridge, MA 02139-4307}
\email{craig@math.mit.edu}
\date{December 20, 2009}

\begin{document}

\begin{abstract}
Using methods of A. Grigor'yan and L. Saloff-Coste we prove that on a manifold with a conical end the heat kernel has
a Gaussian bound.   This result is applied to asymptotically conical K\"{a}hler manifolds.
It is a result of the author and R. Goto that a crepant resolution $\pi:Y\rightarrow X$ of a Ricci-flat K\"{a}hler cone $X$ admits a Ricci-flat K\"{a}hler metric asymptotic to the cone metric in every K\"{a}hler class.   We prove the sharp rate of convergence of the metric to the cone metric.  For compact K\"{a}hler classes this is the same as for the Ricci-flat ALE metrics of P. Kronheimer and D. Joyce.
\end{abstract}

\maketitle


\section{Introduction}

This article considers Riemannian manifolds with a conical end.  A metric cone is a manifold $(C(S),g)$
with $C(S)=\R_{>0} \times S$ and $g=dr^2 +r^2 g_S$.  We consider manifolds with an end for which the metric is approximated by a cone
metric.  This can be considered a generalization of ALE manifolds.  But in this case the compact manifold $(S,g_S)$ is
arbitrary.  We will also consider K\"{a}hler manifolds with an end approximated by a K\"{a}hler cone. 
If $(C(S),g)$ is K\"{a}hler, then $(S,g_S)$ is a Sasaki manifold.  So in this case $(S,g_S)$ is far from arbitrary,
and such manifolds have been studied extensively~\cite{BG3}.  In particular we will consider Ricci-flat asymptotically
conical K\"{a}hler manifolds.    

Ricci-flat K\"{a}hler manifolds with an asymptotically conical end, and in particular resolutions of a Ricci-flat K\"{a}hler cone
$(C(S),g)$, have been of particular interest recently due to their relevance to the AdS/CFT correspondence.   See~\cite{MS2} 
for the construction of many explicit examples which are resolutions of Ricci-flat K\"{a}hler cones.  See also~\cite{MS3,MS4} 
for more on the relevance of these manifolds to AdS/CFT.

Asymptotically conical Ricci-flat K\"{a}hler manifolds have been extensively studied by solving the Monge-Amp\'{e}re equation
in various cases.  There are the existence results of G. Tian and S.-T. Yau~\cite{TY2}, and independently S. Bando and R. Kobayashi~\cite{BK2},
on Ricci-flat K\"{a}hler metrics on quasi-projective manifolds $X\setminus D$.  And there are the results of D. Joyce on the 
existence of Ricci-flat ALE metrics~\cite{Joy1}.  The author~\cite{vC3} and R. Goto~\cite{Got} have proved existence results
on resolutions of Ricci-flat K\"{a}hler cones.  This article shows what solutions to the Monge-Amp\'{e}re equation behave
essentially as in the ALE case due to D. Joyce.  In particular, we have the following for resolutions of cones.
\begin{thm}\label{thm:main}
Let $\pi: \hat{X}\rightarrow X=C(S)\cup\{o\}$ be a crepant resolution of a Ricci-flat K\"{a}hler cone $(X,\ol{g})$ of complex
dimension $m$.  Then in every K\"{a}hler class in $H_c^2(\hat{X},\R)\subset H^2(\hat{X},\R)$ there is a unique Ricci-flat
K\"{a}hler metric $g$ asymptotic to $\ol{g}$ as follows.  There exists an $R>0$ such that for each $k\geq 0$
\begin{equation}\label{eq:conv-cpt}
|\nabla^k(\pi_* g-\ol{g})| =O(r^{-2m-k})\quad\text{on }\{x\in C(S): r(x)>R\},
\end{equation}
where $\nabla$ and the point-wise norm are with respect to $\ol{g}$.

Furthermore, in every K\"{a}hler class in $H^2(\hat{X},\R)\setminus H_c^2(\hat{X},\R)$ there is a Ricci-flat metric $g$
asymptotic to $\ol{g}$.  In this case there exists an $R>0$ such that for each $k\geq 0$
\begin{equation}\label{eq:conv-noncpt}
|\nabla^k(\pi_* g-\ol{g})| =O(r^{-2-k})\quad\text{on }\{x\in C(S): r(x)>R\}.
\end{equation}
\end{thm}
\begin{remark}
Both rates of convergence in (\ref{eq:conv-cpt}) and (\ref{eq:conv-noncpt}) are sharp.  The contribution of this article is
in proving the sharp convergence in (\ref{eq:conv-cpt}).  The author proved a weaker version in~\cite{vC3} where the exponent in
(\ref{eq:conv-cpt}) is $-2m+\delta,\ \delta>0$.  This was also proved by R. Goto~\cite{Got} along with a clever proof of the existence for non-compact K\"{a}hler classes.
\end{remark}

The first part of this article considers the more general case of arbitrary real manifolds with a conical end.
Using methods of A. Grigor'yan and L. Saloff-Coste we prove that every such manifold satisfies the parabolic Harnack
inequality, and as a consequence we have a Gaussian bound on the heat kernel.  This follows from the
scale-invariant Poincar\'{e} inequality, and the volume doubling condition, which is proved
using a discretization technique.  We use this to prove some results on the Laplacian on weighted H\"{o}lder
spaces that will be needed later.

In the second part we consider asymptotically conical K\"{a}hler manifolds.
The main result is a version of the Calabi conjecture for asymptotically conical K\"{a}hler manifolds.
The proof due to D. Joyce~\cite{Joy1} for ALE K\"{a}hler manifolds goes through as is using the Sobolev inequality and
the results on the Laplacian on weighted H\"{o}lder space in this context.  The result is that solutions to the 
Monge-Amp\'{e}re equation on asymptotically conical K\"{a}hler manifolds behave as on ALE K\"{a}hler manifolds.

The final section gives an overview of some examples of Ricci-flat asymptotically conical manifolds given by
Theorem~\ref{thm:main}.  One is considering cones over Sasaki-Einstein manifolds which have crepant resolutions.
Examples can be found from hypersurface singularities.  Many such examples of Sasaki-Einstein manifolds are known
(cf.~\cite{BG2,BGJ,BG3}).  We also briefly discuss the toric case.  The existence problem of Sasaki-Einstein metrics in this case
is solved~\cite{FOW}.  Therefore it is an easy source of examples.

Theorem~\ref{thm:main} does not settle the existence problem of asymptotically conical Ricci-flat K\"{a}hler metrics.
More generally, one can consider crepant resolutions $\pi:\hat{X}\rightarrow X$ of more general affine varieties $X$
such that $\hat{X}$ is a quasi-projective variety with an end which is diffeomorphic to a cone, but not holomorphically.

\subsection*{Acknowledgements}

I would like to thank Gilles Carron for referring me to the articles~\cite{Min} of V. Minerbe and~\cite{GriSal-Cos} of
A. Grigor'yan and L. Saloff-Coste which allowed me to prove the heat kernel bounds.

\subsection*{Notation}

We will use `$\subset$' to denote set inclusion, proper and otherwise.  The geodesic ball
centered at $x\in M$ of radius $r$ will be denoted by $B(x,r)$, and we will use $V(x,r)$ to denote its volume with respect
to the Riemannian measure.

\section{Asymptotically conical manifolds and}

\subsection{Asymptotically conical manifolds}

We cover some analysis on asymptotically conical manifolds.  This will be used in the sequel in the proof of the
Calabi conjecture.  But it may be of independent interest.

Let $(X,g)$ be a complete Riemannian manifold.  Using the metric $g$ and Levi-Civita connection $\nabla$ we
have the Sobolev spaces $L_k^q(X)$, the $C^k$-spaces $C^k(X)$, and the H\"{o}lder spaces $C^{k,\alpha}(X)$.
But we will also need \emph{weighted Sobolev spaces} and \emph{weighted H\"{o}lder spaces}.

Let $o\in X$ be a fixed point and $d(o,x)$ the Riemannian distance to $x\in X$.  Then we define a weight
function $\rho(x)=(1+d(o,x)^2)^{1/2}$.
\begin{defn}
Let $q\geq 1$, $\beta\in\R$, and $k$ a nonnegative integer.  We define the \emph{weighted Sobolev space}
$L_{k,\beta}^q$ to be the set of functions $f$ on $X$ which are locally integrable and whose weak derivatives
up to order $k$ are locally integrable and for which the norm
\[ \|f\|_{L_{k,\beta}^q}:=\left(\sum_{j=0}^k \int_X |\rho^{j-\beta}\nabla^j f|^q \rho^{-n}d\mu_g \right)^{1/q}\]
is finite.  Then $L_{k,\beta}^q$ is a Banach space with this norm.
\end{defn}

\begin{defn}
For $\beta\in\R$ and $k$ a nonnegative integer we define $C_\beta^k(X)$ to be the space of continuous functions $f$
with $k$ continuous derivatives for which the norm
\[ \|f\|_{C_\beta^k} := \sum_{j=0}^k \underset{X}{\sup}|\rho^{j-\beta}\nabla^j f|\]
is finite.  Then $C_\beta^k(X)$ is a Banach space with this norm.

Let $\inj(x)$ be the injectivity radius at $x\in X$, and $d(x,y)$ the distance between $x,y\in X$.  Then for
$\alpha,\gamma\in\R$ and $T$ a tensor field define
 \[ [T]_{\alpha,\gamma} :=\sup_{\substack{x\neq y\\ d(x,y)<\inj(x)}}\left[\min(\rho(x),\rho(y))^{-\gamma}\cdot\frac{|T(x)-T(y)|}{d(x,y)^\alpha}\right],  \]
where $|T(x)-T(y)|$ is defined by parallel translation along the unique geodesic between $x$ and $y$.

For $\alpha\in(0,1)$ define the \emph{weighted H\"{o}lder space} $C_\beta^{k,\alpha}(X)$ to be the set of
$f\in C_\beta^k(X)$ for which the norm
\begin{equation}
\|f\|_{C_\beta^{k,\alpha}} :=\|f\|_{C_\beta^k} +[\nabla^k f]_{\alpha,\beta-k-\alpha}
\end{equation}
is finite.  Then $C_\beta^{k,\alpha}(X)$ is a Banach space with this norm.

Define $C_\beta^\infty (X)$ to be the intersection of the $C_\beta^k(X)$, for $k\geq 0$.
\end{defn}
It will be convenient to use a different weight function $\rho(x)$ on the manifolds we consider but it will define
equivalent norms.  See~\cite{Ch-SCh-B} for an introduction to the theory of weighted H\"{o}lder spaces.  

\begin{defn}\label{defn:cone}
Let $(S,g_S)$ be a compact Riemannian manifold.  The cone over $(S,g_S)$ is the Riemannian manifold $(C(S),g)$
with $C(S)=\R_{>0} \times S$ and $g=dr^2 +r^2 g_S$ where $r$ usual coordinate on $R_{>0}$.
\end{defn}
We will sometimes consider $(C(S),g)$ with the apex $o\in C(S)$ at $r=0$.  This is singular at $o\in C(S)$
unless $S=S^{n-1}$ is the sphere with the round metric.

\begin{defn}\label{defn:AC}
Let $(C(S),g_0)$ be a metric cone.  Then $(X,g)$ is asymptotically conical of order $(\delta, k+\alpha)$,
if there is compact subset $K\subset X$, a compact neighborhood $o\in K_0 \subset C(S)$, and diffeomorphism
$\phi: X\setminus K \rightarrow C(S)\setminus K_0$ so that
\[ |\phi_* g -g_0| \in C_{\delta} ^{k,\alpha}\quad\text{on } C(S)\setminus K_0.\]
We will abbreviate this by $\AC(\delta, a+\alpha)$, and always $\delta<0$.
\end{defn}

Suppose $K_0\subset D_{r_0} =\{(r,s)\in C(S): r< r_0\}$, and assume $r_0 \geq 2$.  Then a smooth extension of
$\phi^* r :X\setminus\phi^{-1}(D_{r_0}) \rightarrow[2,\infty)$ to $\rho: X\rightarrow [1,\infty)$ is a
\emph{radius function} of the $\AC$ manifold $(X,g)$.

\begin{remark}
It will be convenient sometimes to define weighted H\"{o}lder and Sobolev spaces using the radius function
as the weight function $\rho$.  In other cases we will use $\tilde{\rho(x)}=(1+d(o,x)^2)^{1/2}$ for $o\in K$.
This will mostly be a matter of convenience.  If $(X,g)$ is $\AC(\delta,0), \delta>0,$ it is not difficult to check
that $c^{-1}\tilde{\rho}\leq\rho\leq c\tilde{\rho}$ for $c>0$.  Thus the weighted norms are equivalent.
We will denote such a relation between functions by $\rho\sim\tilde{\rho}$.
\end{remark}

\subsection{The Sobolev inequality}

We give a proof of the Sobolev inequality on asymptotically conical manifolds. 
Recall that metrics $g$ and $\tilde{g}$ on $X$ are \emph{quasi-isometric} if there is a $c>0$ so that
\begin{equation}\label{eq:quasi-isom}
c^{-1}g_x (W,W)\leq \tilde{g}_x (W,W)\leq cg_x (W,W),\quad\forall x\in X\text{ and }\forall W\in T_x X.
\end{equation}
\begin{thm}\label{thm:Sobolev}
Let $(X,g)$ be AC of order $(\delta, 0), \delta<0,$ or merely quasi-isometric to an AC manifold.  Then there is a constant
$C>0$ so that we have the \emph{Sobolev inequality}
\begin{equation}\label{eq:Sobolev}
\|f\|_{n/(n-1)} \leq C\|\nabla f\|_1 ,\quad\forall f\in C_0 ^\infty (X).
\end{equation}
\end{thm}

And easy argument with the H\"{o}lder inequality gives the following.
\begin{cor}\label{cor:Sobolev}
For any real $p, 1\leq p<n$ we have
\begin{equation}
\|f\|_{np/(n-p)} \leq C\|\nabla f\|_p ,\quad\forall f\in C_0 ^\infty (X).
\end{equation}
\end{cor}

Note that (\ref{eq:Sobolev}), and most of what follows in this section, is stable under quasi-isometries.
So all that is essential in Theorem~\ref{thm:Sobolev} is that the end of $(X,g)$ is quasi-isometric to a cone.

We prove Theorem~\ref{thm:Sobolev} with a discretization procedure used in~\cite{GriSal-Cos} to prove
Poincar\'{e} inequalities and generalized in~\cite{Min} to prove more general Poincar\'{e}-Sobolev inequalities.
A proof of (\ref{eq:Sobolev}) is given in~\cite{TY2}, but our result is more general and the proof is simpler.
Let $\mu$ denote the Riemannian measure on $(X,g)$.
\begin{defn}\label{defn:good-cov}
Let $A\subset A^{\#}$ be subsets of $X$.  A family $\mathcal{U}=(U_i ,U^*_i ,U^{\#}_i)_{i\in I}$ consisting of subsets
of $X$ having finite measure is said to be a good covering of $A$ in $A^{\#}$ if the following is true:
\begin{thmlist}
\item  $A\subset\cup_i U_i \subset\cup_i U_i^{\#} \subset A^{\#}$;\label{defn:gc-i}
\item  $\forall i\in I, U_i \subset U_i^* \subset U_i^{\#}$;\label{defn:gc-ii}
\item  $\exists Q_1 ,\forall i_0 \in I, \Card\{i\in I :U_{i_0}^{\#} \cap U_i^{\#} \neq\emptyset\}\leq Q_1$;\label{defn:gc-iii}
\item  For every $(i,j)\in I^2$ satisfying $\ol{U}_i \cap\ol{U}_j \neq\emptyset$, there is an element $k(i,j)$ such that
$U_i \cup U_j \subset U^*_{k(i,j)}$;\label{defn:gc-iv}
\item  There exists a $Q_2$ such that for every $(i,j)\in I^2$ with $\ol{U}_i \cap\ol{U}_j \neq\emptyset$ we have
$\mu(U^*_{k(i,j)})\leq Q_2 \min(\mu(U_i),\mu(U_j))$.\label{defn:gc-v}
\end{thmlist}
\end{defn}

Given a Borel set $U$ with finite $\mu$-measure and a $\mu$-integrable function $f$, denote by $f_U$ the mean
value of $f$ on $U$:
\[ f_U =\frac{1}{\mu(U)}\int_U f d\mu.\]

Given any good covering we have an associated weighted graph $(\mathcal{G},m)$ as follows.
\begin{defn}\label{defn:weig-graph}
Let $\mathcal{U}=(U_i ,U^*_i ,U^{\#}_i)_{i\in I}$ be a good covering of $A$ in $A^{\#}$.  The associated weighted
graph $(\mathcal{G},m)$ has vertices $\mathcal{V}=I$ and edges
$\mathcal{E}=\{\{i,j\}\subset\mathcal{V}: i\neq j, \ol{U}_i \cap\ol{U}_j \neq\emptyset\}$.

Measures, both denoted $m$, are defined on $\mathcal{V}$ and $\mathcal{E}$ as follows:
\begin{thmlist}
\item  $\forall i\in\mathcal{V}, m(i)=\mu(U_i)$;
\item  $\forall \{i,j\}\in\mathcal{E}, m(i,j)=\max(m(i),m(j)).$
\end{thmlist}
\end{defn}

We will patch together Sobolev inequalities on the subsets $(U_i ,U^*_i ,U^{\#}_i)$ of a good covering using
discrete Sobolev inequalities on the associated weighted graph $(\mathcal{G},m)$.

\begin{defn}\label{defn:Sobolev-cont}
Let $1\leq p<\infty$, and suppose $p<\nu\leq\infty$.  We say that a good covering $\mathcal{U}$ satisfies a continuous
$L^p$ Sobolev inequality of order $\nu$ if there exists a constant $S_c$ such that for every $i\in I$
\[ \left(\int_{U_i} |f-f_{U_i}|^{\frac{p\nu}{\nu-p}}d\mu\right)^{\frac{\nu-p}{\nu}} \leq S_c \int_{U^*_i}|\nabla f|^p d\mu, \quad\forall f\in C^{\infty}(U^*_i),\]
and
\[ \left(\int_{U^*_i} |f-f_{U_i}|^{\frac{p\nu}{\nu-p}}d\mu\right)^{\frac{\nu-p}{\nu}} \leq S_c \int_{U^{\#}_i}|\nabla f|^p d\mu, \quad\forall f\in C^{\infty}(U^{\#}_i).\]
\end{defn}

\begin{defn}\label{defn:Sobolev-dis-Dir}
Given $1\leq p<\nu\leq\infty$, we say that the weighted graph $(\mathcal{G},m)$ satisfies a discrete $L^p$
Sobolev-Dirichlet inequality of order $\nu$ if there exists a constant $S_d$ such that for every
$f\in L^p(\mathcal{V},m)$
\[ \left(\sum_{i\in\mathcal{V}} |f(i)|^{\frac{p\nu}{\nu-p}} m(i)\right)^{\frac{\nu-p}{\nu}}\leq S_d \sum_{{i,j}\in\mathcal{E}} |f(i)-f(j)|^p m(i,j). \]
\end{defn}

Suppose $(\mathcal{G},m)$ is a finite weighted graph.  Then for $f\in\R^\nu$, denote by $m(f)$ the mean of $f$
\[m(f)=\frac{1}{m(\mathcal{V})}\sum_{i\in\mathcal{V}} f(i).\]

\begin{defn}\label{defn:Sobolev-dis-Neu}
Given $1\leq p<\nu\leq\infty$, we say that a finite weighted graph $(\mathcal{G},m)$ satisfies a discrete $L^p$
Sobolev-Neumann inequality of order $\nu$ if there exists a constant $S_d$ such that for every $f\in\R^{N}$, $N=|\mathcal{V}|$,
\[ \left(\sum_{i\in\mathcal{V}} |f(i)-m(f)|^{\frac{p\nu}{\nu-p}} m(i)\right)^{\frac{\nu-p}{\nu}}\leq S_d \sum_{{i,j}\in\mathcal{E}} |f(i)-f(j)|^p m(i,j). \]
\end{defn}

Note that the $L^p$ Sobolev inequality of order $\nu=\infty$ is the $L^p$ Poincar\'{e} inequality.

The following theorem is crucial in the proof of the Sobolev inequality on $(X,g)$.  We also state the ``Neumann''
analogue of this theorem since we will use it in a proof of a Poincar\'{e} inequality on $(X,g)$.

\begin{thm}[\cite{Min}]\label{thm:cov-Sobolev-Dir}
Suppose $1\leq p<\nu\leq\infty$.  If a good covering $\mathcal{U}$ of $A$ in $A^{\#}$ satisfies the continuous $L^p$ Sobolev inequality of order $\nu$ (Definition~\ref{defn:Sobolev-cont}) and the discrete $L^p$ Sobolev-Dirichlet
inequality of order $\infty$ (Definition~\ref{defn:Sobolev-dis-Dir}), then the following Sobolev-Dirichlet inequality
is true:
\[ \left(\int_A |f|^{\frac{p\nu}{\nu-p}} d\mu\right)^{\frac{\nu-p}{p}}\leq S\int_{A^{\#}}|\nabla f|^p d\mu,\quad\forall f\in C_c^{\infty}(A).  \]
Furthermore, one can choose $S=S_c Q_1 2^{p-1+\frac{p}{\nu}}(1+S_d Q_2(2^p Q_1^2)^{\frac{\nu}{\nu-p}})^{\frac{\nu-p}{\nu}}$.
\end{thm}

\begin{thm}[\cite{Min}]\label{thm:cov-Sobolev-Neu}
Suppose $1\leq p<\nu\leq\infty$.  If a finite good covering $\mathcal{U}$ of $A$ in $A^{\#}$ satisfies the continuous $L^p$ Sobolev inequality of order $\nu$ (Definition~\ref{defn:Sobolev-cont}) and the discrete $L^p$ Sobolev-Neumann
inequality of order $\infty$ (Definition~\ref{defn:Sobolev-dis-Neu}), then the following Sobolev-Neumann inequality
is true:
\[ \left(\int_A |f-f_A|^{\frac{p\nu}{\nu-p}} d\mu\right)^{\frac{\nu-p}{p}}\leq S\int_{A^{\#}}|\nabla f|^p d\mu,\quad\forall f\in C^{\infty}(A^{\#}).  \]
Furthermore, one can choose $S=S_c Q_1 2^{2p-1+\frac{p}{\nu}}(1+S_d Q_2(2^p Q_1^2)^{\frac{\nu}{\nu-p}})^{\frac{\nu-p}{\nu}}$.
\end{thm}

We will need prove discrete Sobolev inequalities to apply these theorems.
The discrete Sobolev-Dirichlet inequality in Definition~\ref{defn:Sobolev-dis-Dir} follows
from an isoperimetric inequality on the graph $(\mathcal{G},m)$.
\begin{defn}
Given a graph $\mathcal{G}$, we define the boundary $\partial\Omega$ of a subset $\Omega\subset\mathcal{V}$ as
\[ \partial\Omega:=\{\{i,j\}\in\mathcal{E}:\{i,j\}\cap\Omega\neq\emptyset\text{ and } \{i,j\}\cap(\mathcal{V}\setminus\Omega)\neq\emptyset\}.\]
\end{defn}

The following result is completely analogous to the situation with continuous Sobolev inequalities.
\begin{prop}\label{prop:disc-isop}
Let $(\mathcal{G},m)$ be an infinite weighted graph and fix $1<\nu\leq\infty$.  Then the $L^1$ Sobolev inequality
of order $\nu$
\[ \left(\sum_{i\in\mathcal{V}} |f(i)|^{\frac{\nu}{\nu-1}} m(i)\right)^{\frac{\nu-1}{\nu}}\leq C\sum_{\{i,j\}\in\mathcal{E}} |f(i)-f(j)|m(i,j),\quad\forall f\in L^1(\mathcal{V},m),\]
is equivalent to the isoperimetric inequality of order $\nu$
\begin{equation}
m(\Omega)^{\frac{\nu-1}{\nu}} \leq C m(\partial\Omega),\quad\forall\Omega\subset\mathcal{V}\text{ with } m(\Omega)<\infty.
\end{equation}
\end{prop}

In the following section we will need a good bound on the constant in the discrete Sobolev-Neumann inequality.
Suppose $(\mathcal{G},m)$ is a finite weighted graph with $|\mathcal{V}|=N$.
The \emph{spectral gap} $\lambda(\mathcal{G},m)$ is defined by
\begin{equation}\label{eq:spec-gap}
\lambda(\mathcal{G},m)=\inf\left\{\frac{\sum_{\{i,j\}\in\mathcal{E}}|f(i)-f(j)|^2 m(i,j)}{\sum_{i\in\mathcal{V}}|f(i)-m(f)|^2 m(i)}: f\in\R^N \setminus\{0\} \right\}.
\end{equation}
Clearly, if $S_d$ is the minimum constant so that the discrete Sobolev-Neumann $L^2$ inequality of order $\infty$
holds, then $S_d= 1/\lambda(\mathcal{G},m)$.

The \emph{Cheeger constant} $h(\mathcal{G},m)$ is defined by
\begin{equation}
h(\mathcal{G},m)=\inf\left\{\frac{\sum_{\{i,j\}\in\mathcal{E}}|f(i)-f(j)| m(i,j)}{\inf_c \sum_{i\in\mathcal{V}}|f(i)-c| m(i)}: f\in\R^N \right\}.
\end{equation}
It is well-known (cf.~\cite{Sal-Cos1}) that the Cheeger constant is equivalent to the isoperimetric inequality
\begin{equation}\label{eq:disc-isop-Neu}
h(\mathcal{G},m)=\inf\left\{\frac{m(\partial U)}{m(U)}: U\subset\mathcal{V}, 0<m(U)\leq\frac{1}{2}m(\mathcal{V})\right\}.
\end{equation}
And we can compare the constants $h(\mathcal{G},m)$ and $\lambda(\mathcal{G},m)$ by
\begin{equation}\label{eq:spect-Cheeger}
\frac{h^2(\mathcal{G},m)}{8m_0}\leq\lambda(\mathcal{G},m)\leq h(\mathcal{G},m),
\end{equation}
where
\[ m_0 :=\underset{i\in\mathcal{V}}{\max}\left\{\frac{1}{m(i)}\sum_{j\in\mathcal{V}}m(i,j)\right\}. \]

We state the Euclidean scale-invariant Sobolev-Poincar\'{e} inequalities.  They will be used in the proof of
Theorem~\ref{thm:Sobolev} and also in the proof of the scale invariant Poincar\'{e} inequality on $(X,g)$.
Let $B_r =B(o,r)\subset\R^n$ be the ball of radius $r>0$.  For $1\leq p\leq n$ and $n\leq\nu\leq\infty$ there exists
a constant $S_{p,\nu}$, only depending on $p,\nu$, so that
\begin{equation}\label{eq:Sobolev-loc}
\left(\int_{B_r}|f-f_r|^{\frac{p\nu}{\nu-p}} d\mu\right)^{\frac{\nu-p}{p\nu}} \leq S_{p,\nu} r^{1-\frac{n}{\nu}}\left(\int_{B_r}|\nabla f|^p d\mu\right)^{\frac{1}{p}},\quad\forall f\in C^{\infty}(B_r).
\end{equation}

We may assume that the end of $(X,g)$ has the conical metric.  That is, if $\rho$ is a radius function, then there
is an $r_0 \geq 2$ so that $\phi: X\setminus K_{r_0}\rightarrow C(S)\setminus\ol{D}_{r_0}$, with
$K_{r_0} =\{x\in X: \rho(x)\leq r_0 \}$ and $\ol{D}_{r_0} =\{x\in C(S): r(x)\leq r_0 \}$, is an isometry.
\begin{lem}\label{lem:Sobolev-anal}
Fix $R \geq r_0 ,\kappa >1$ and consider the annulus $A=A(R,\kappa R)= D_{\kappa R} \setminus D_{R}$.  Then if we let
$A_\delta$ be the $\delta R$-neighborhood of $A$, with $0<\delta<1$ sufficiently small, there is a constant
$C>0$ independent of $R$ so that
\begin{equation}\label{eq:Sobolev-anal}
\left(\int_A |f-f_A|^{\frac{n}{n-1}} d\mu\right)^{\frac{n-1}{n}} \leq C \int_{A_\delta}|\nabla f|d\mu,\quad\forall f\in C^{\infty}(A_\delta).
\end{equation}
\end{lem}
\begin{proof}
Set $s=\delta R$, and let $\{x_i\}_{i\in I}$ be a maximal subset of $A$ such that the distance between any two of its
elements is at least $s$.  Set $V_i =B(x_i ,s)$ and $V_i^* =V_i^{\#}=B(x_i,3s)$.  Then $(V_i,V_i^*,V_i^{\#})$ is
a finite good covering of $A$ in $A_\delta$.  Conditions (\ref{defn:gc-iii}) and (\ref{defn:gc-v}) are satisfied with
$Q_1 =\Card(I)$ and $Q_2 =\frac{\max(\mu(V_i^*) :i\in I)}{\min(\mu(V_i):i\in I)}$.  And in (\ref{defn:gc-iv}) we may
take $k(i,j)=i$.

We will apply Theorem~\ref{thm:cov-Sobolev-Neu}.  The exponential map gives coordinates $\phi_i :B_{s} \rightarrow V_i$
and $\phi_i^* :B_{3s}\rightarrow V_i^*$.  Then if $g_0$ is the flat metric on $B_{3s}$, by finiteness there is a $c>0$ so that $c^{-1}g_0 \leq (\phi_i^*)^* g\leq cg_0$ for all $i\in I$.

We will need the following which is easy to prove using the H\"{o}lder inequality.
Let $(U,\lambda)$ be a finite measure space, then
\begin{equation}\label{eq:aver-min}
\|f-f_{U,\lambda}\|_{L^q(U,\lambda)} \leq 2\underset{c\in\R}{\inf}\|f-c\|_{L^q(U,\lambda)}.
\end{equation}
Then (\ref{eq:Sobolev-loc}) with $p=1$ and $\nu=n$, (\ref{eq:aver-min}), and the uniform bound on $g$ imply that for
$V=V_i$ or $V_i^*$
\begin{equation}
\left(\int_V |f-f_{V,\mu}|^{\frac{n}{n-1}} d\mu\right)^{\frac{n-1}{n}} \leq 2 S_{1,n} c^n \int_V |\nabla f|d\nu.
\end{equation}
Thus the covering $(V_i,V_i^*,V_i^{\#})$ satisfies the continuous $L^1$ Sobolev inequality of order $\nu=n$.
It remains to show the discrete $L^1$ Sobolev-Neumann inequality.  But this follows because $\mathcal{G}$ is a
finite connected graph, and any two norms on a finite dimensional vector space are equivalent.
This proves (\ref{eq:Sobolev-anal}).  Then one can see that the same $C$ can be used in (\ref{eq:Sobolev-anal}) for
any $R\geq r_0$ by considering the Euler action $\psi_a :C(S)\rightarrow C(S), a>0,$ which acts by homotheties.
We have $\psi_a :A(R,\kappa R)\rightarrow A(aR,a\kappa R)=A(R',\kappa R')$.  And if $A'_\delta$ denotes the
$\delta R'$-neighborhood of $A(R',\kappa R')$, then $\psi_a :A_\delta \rightarrow A'_\delta$ is a homothetic diffeomorphism.  Since (\ref{eq:Sobolev-anal}) is invariant under homotheties, the lemma follows.
\end{proof}

\begin{proof}[Proof of Theorem~\ref{thm:Sobolev}]
Let $R>0$ and $\kappa>1$ be as above and define $R_i =\kappa^i R$.  We define a good covering of $(X,g)$.
Define $A_i =A(R_{i-1},R_i)$ for $i\geq 1$, and $A_0 =D_R =\{x\in X: \rho(x)<R\}$.  We let $A^*_i$ be the union of
all the $A_j$ whose closure intersects $\ol{A}_i$.  And similarly, $A^{\#}_i$ is the union of all the $A_j$ whose closure
intersects $\ol{A}^*_i$.

Lemma~\ref{lem:Sobolev-anal} shows that the continuous $L^1$ Sobolev inequality of order $n$ holds for the pairs
$(A_i,A^*_i)$ with $i\geq 1$, and for $(A^*_i,A^{\#}_i)$ with $i\geq 2$.  Since $A_0 =D_R$ is pre-compact one can
show that (\ref{eq:Sobolev-anal}) holds with $A=A_0$ using the same argument in Lemma~\ref{lem:Sobolev-anal}. Thus the
continuous $L^1$ Sobolev inequality holds for $(A_0,A^*_0)$, and the remaining cases are identical.

Proposition~\ref{prop:disc-isop} shows that the discrete $L^1$ Sobolev-Dirichlet inequality, of order $\infty$, holds if
\begin{equation}\label{eq:disc-isop}
\frac{m(\Omega)}{m(\partial\Omega)}\leq C,\quad\forall\Omega\subset\mathcal{V}\text{ with } m(\Omega)<\infty.
\end{equation}
Suppose $\Omega\subset\{0,\ldots,j\}$ with $j\in\Omega$.  Then
\[ \begin{split}
m(\Omega) & \leq\sum_{i=0}^j m(i) \\
          & =\mu(A_0) +\sum_{i=1}^j \mu(A(\kappa^{i-1}R,\kappa^i R))\\
          & =\mu(A_0)+ \mu(A_1)\sum_{i=1}^j \kappa^{n(i-1)}\\
          & =\mu(A_0) +\mu(A_1)\frac{\kappa^{nj} -1}{\kappa^n -1}
\end{split}\]
And
\[ \mu(\partial\Omega)\geq m(j+1)=\mu(A_1)\kappa^{nj}. \]
The isoperimetric inequality (\ref{eq:disc-isop}) holds with $C=\frac{\mu(A_0)}{\mu(A_1)} +\frac{1}{\kappa^n -1}$.
We then apply Theorem~\ref{thm:cov-Sobolev-Dir}, and the prove is complete.
\end{proof}

Let $V(x,r)$ be the volume of the geodesic ball $B(x,r)$ of radius $r$ centered at $x\in X$.  It is well known
(cf.~\cite{Sal-Cos2}) that (\ref{eq:Sobolev}) implies the following volume growth condition.
\begin{cor}\label{cor:Sobolev-vol}
There is a constant $c>0$, depending only on $n$ and $C$ in (\ref{eq:Sobolev}), so that $V(x,r)\geq cr^{n}$.
\end{cor}
\begin{remark}\label{rem:Sob}
The above arguments can easily be adapted to prove the Sobolev inequality (\ref{eq:Sobolev}) on an $\AC$ manifold with
multiple, but finitely many, ends.
\end{remark}

\subsection{Gaussian bound on the heat kernel}

Recall that the \emph{heat kernel} $h(t,x,y)$ is a smooth function on $\R_{>0}\times X\times X$ symmetric in $x,y$ which is the fundamental solution to the heat diffusion equation, $(\partial_t +\Delta_x)h=0$, with $\underset{t\rightarrow 0^+}{\lim}h(t,x,y)=\delta_x$.
In this section we prove that $h(t,x,y)$ satisfies a Gaussian bound.  First we need a definition.

\begin{defn}\label{defn:Poin}
Let $U\subset U'$ be relatively compact open subsets of $(X,g)$.  The \emph{Poincar\'{e} constant} of the pair
$(U,U')$ is the smallest positive number $\Lambda(U,U')$ such that
\begin{equation}
\int_U |f-f_U|^2 d\mu\leq\Lambda(U,U')\int_{U'} |\nabla f|^2 d\mu, \quad\forall f\in C^{\infty}(U').
\end{equation}
We say that $(X,g)$ satisfies a \emph{Poincar\'{e} inequality} with \emph{parameter} $0<\delta\leq 1$ if there is
a constant $C_P >0$ so that for any ball $B(x,r)$  
\begin{equation}\label{eq:Poin}
\Lambda(B(x,\delta r),B(x,r))\leq C_P r^2.
\end{equation}
\end{defn}

\begin{thm}\label{thm:para-Harn}
For a complete manifold $(X,g)$ the following are equivalent.
\begin{thmlist}
\item  $(X,g)$ satisfies the \emph{volume doubling} condition.  That is,\label{thm:item:vd-Poin}
there exists a constant $C_D >0$ so that for any ball $B(x,r)$
\begin{equation}\label{eq:vol-doub}
V(x,2r)\leq C_D V(x,r).
\end{equation}
And $(X,g)$ satisfies a Poincar\'{e} inequality with parameter $0<\delta\leq 1$.

\item  The \emph{heat kernel} $h(t,x,y)$ of $(X,g)$ satisfies the two-sided Gaussian bound\label{thm:item:Gauss}
\begin{equation}\label{eq:Gauss}
\frac{c_1 \exp(-C_1 d^2(x,y)/t)}{V(x,\sqrt{t})}\leq h(t,x,y)\leq\frac{C_2 \exp(-c_2 d^2(x,y)/ t)}{V(x,\sqrt{t})}.
\end{equation}
\end{thmlist}
\end{thm}

We also have the following time derivative estimates on $h(t,x,y)$ if the equivalent conditions in the theorem hold.
For any integer $k$,
\begin{equation}
|\partial_t^k h(t,x,y)|\leq \frac{A_k \exp(-c_2 d^2(x,y)/t)}{t^k V(x,\sqrt{t})},
\end{equation}
and there is an $\epsilon>0$ so that
\begin{equation}\label{eq:Gaus-der}
|\partial_t^k (t,x,y)-\partial_t^k h(t,x,z)|\leq \frac{d(y,z)^\epsilon \exp(-c_2 d(x,y)/t)}{t^{k+\frac{\epsilon}{2}} V(x,\sqrt{t})},
\end{equation}
for $d(y,z)\leq\sqrt{t}$.

\begin{remark}
It is known~\cite{Jer} that if a Poincar\'{e} inequality (\ref{eq:Poin}) holds for $0<\delta<1$ then it holds
for all $\delta\in (0,1]$.
\end{remark}

A. Grigor'yan~\cite{Gri} and L. Soloff-Coste~\cite{Sal-Cos} proved that (\ref{thm:item:vd-Poin}) is equivalent to a 
parabolic Harnack inequality.  The equivalence of (\ref{thm:item:Gauss}) and the parabolic Harnack inequality goes back to~\cite{FabStro}.
See also~\cite{Sal-Cos2} for proofs of both of these equivalences.

\begin{remark}\label{rem:vd-Poin-quasi}
It is not difficult to check that both the volume doubling condition and the existence of a Poincar\'{e} inequality
in Theorem~\ref{thm:para-Harn}.\ref{thm:item:vd-Poin} are invariant under quasi-isometry.  This will be used to
simplify the proofs below.
\end{remark}

The main result of this section is the following.
\begin{thm}\label{thm:Poin-con}
Suppose $(X,g)$ is $\AC$ or merely quasi-isometric to an $\AC$ metric.  Then the equivalent conditions of Theorem
\ref{thm:para-Harn} hold on $(X,g)$.
\end{thm}

We will prove Theorem~\ref{thm:Poin-con} by showing that $(X,g)$ satisfies the volume doubling (\ref{eq:vol-doub}) and
the Poincar\'{e} inequality (\ref{eq:Poin}).  We will need some definitions for the proof of Theorem~\ref{thm:Poin-con}.
\begin{defn}
Fix $o\in X$ and a parameter $0<\epsilon\leq 1$ (the \emph{remote parameter}).
\begin{thmlist}
\item  We say that a ball $B(x,r)$ is \emph{remote} if $r\leq\epsilon\frac{1}{2}d(o,x)$.
\item  We say that a ball $B(o,r)$ is \emph{anchored}.
\end{thmlist}
\end{defn}

\begin{lem}\label{lem:v-d}
Let $(X,g)$ be $\AC$.  Then $(X,g)$ satisfies the volume doubling condition (\ref{eq:vol-doub}).
\end{lem}
\begin{proof}
Per Remark~\ref{rem:vd-Poin-quasi} we may assume $(X,g)$ is conical outside a compact set.
We may assume that $\rho$ is a radius function and there
is an $r_0 \geq 2$ so that $\phi: X\setminus D_{r_0}\rightarrow C(S)\setminus K_{r_0}$, with
$D_{r_0} =\{x\in X: \rho(x)\leq r_0 \}$ and $K_{r_0} =\{x\in C(S): r(x)\leq r_0 \}$, is an isometry.

We first prove that volume doubling is satisfied for remote balls.
Choose $\delta<\inj(X,g)$ and $R>r_0$ so that $R-\delta >r_0$.  Let $\{x_i \}_{i\in I}$ be a maximal set of points in
$\ol{D}_R$ such that the distance between any two is at least $\delta$.
For each $i\in I$ let $\mathcal{B}_i \subset T B(x_i,\delta)$ be the radius $\delta$ disk bundle.
And fix an isomorphism $\beta_i :B_\delta \times B(x_i,\delta)\cong\mathcal{B}_i$ linear and preserving distances on the
fibers.  Then define maps $\psi_i : B_\delta \times B(x_i,\delta) \rightarrow X$ by $\psi_i (w,x)=\exp_x (\beta_i(w,x))$.
Let $g_i$ be the restriction of of $\psi_i^*g$ to the fibers on $B_\delta \times B(x_i,\delta)$.  By an easy compactness
argument it is easy to see that if $g_0$ is the flat metric on $B_\delta$, then $c^{-1}g_0 \leq g_i \leq cg_0$ for $c>1$ uniformly in $B(x_i,\delta)$ for each $i\in I$.  If $x\in\ol{D}_R$, then the above arguments show that $g$ in the chart
$\exp_x :B_\delta \rightarrow B(x,\delta)$ satisfies $c^{-1}g_0 \leq g \leq cg_0$.
Then one can show that the volume doubling condition holds for balls $B(x,r)$ with $x\in\ol{D}_R$ and
$r<\frac{1}{2}\delta$, i.e. there exists a $C>0$ so that $V(x,2r)\leq C V(x,r)$.

Now let $x\in X$ with $\rho(x)>R$.  Then by applying the Euler action $\psi_a :C(S)\rightarrow C(S)$ to the above
charts, we have the chart $\exp_x :B_{\frac{\rho(x)}{R}\delta} \rightarrow B(x,\frac{\rho(x)}{R}\delta)$ in which
$c^{-1}g_0 \leq g \leq cg_0$.  Thus we have volume doubling for balls $B(x,r)$ with $r<\frac{\rho(x)}{2R}\delta$.
Since $\rho\sim d(o,-)$, volume doubling holds for remote balls if we chose a small enough remote parameter $\epsilon>0$.

There is a constant $C_o >0$ so that $V(o,r)\leq C_o r^n$.
And by Corollary~\ref{cor:Sobolev} there is a constant $C>0$ so that $V(x,r)\geq Cr^n$ for any $x\in X$.

Choose $\epsilon$ to be the remote parameter from above.  Set $d(o,x)=\ell$.  We consider three cases.

\emph{Case 1:}  If $r\leq\frac{1}{2}\epsilon\ell$, then the ball $B(x,r)$ is remote and volume doubling holds.

\emph{Case 2:}  If $r\geq\frac{3}{2} \ell$, then we have
\[ V(x,2r)\leq V(o,\frac{8}{3}r)\leq C_o \frac{8^n}{3^n}r^n \leq\frac{8^n C_o}{3^n C} V(x,r).\]

\emph{Case 3:}  If $\frac{1}{2}\epsilon\ell\leq r\leq\frac{3}{2} \ell$, then
\[ V(x,2r)\leq V(o,8\ell)\leq C_o 8^n \ell^n \leq \frac{C_o}{C}\frac{16^n}{\epsilon^n}V(x,r).\]
\end{proof}

The proof of the following is straight forward.
\begin{lem}[\cite{GriSal-Cos}]\label{lem:rem-anch}
Given $o\in X$ suppose the Poincar\'{e} inequality holds for all remote and anchored balls with parameter
$0<\delta_0 \leq 1$ and constant $C_P >0$.  That is, for any remote or anchored ball $B(x,r)$
\[ \Lambda(B(x,\delta_0 r),B(x,r))\leq C_P r^2.\]
Then the Poincar\'{e} inequality holds for any ball with parameter $\delta=\epsilon\delta^2_0 /8$ and constant
$C_P>0$.  That is, for any ball $B(x,r)$
\[   \Lambda(B(x,\delta r),B(x,r))\leq C_P r^2.\]
\end{lem}

\begin{proof}[Proof of Theorem~\ref{thm:Poin-con}.]
We will use the same notation used in the proof of Lemma~\ref{lem:v-d}.
From the proof of Lemma~\ref{lem:v-d} we have a remote parameter $\epsilon>0$ so that for $r\leq\epsilon\frac{1}{2}d(o,x)$
in the chart $\exp_x :B_\delta \rightarrow B(x,r)$ the metric $g$ satisfies $c^{-1}g_0 \leq g \leq cg_0$.
Then by the local Poincar\'{e} inequality, (\ref{eq:Sobolev-loc}) with $p=2, \nu=\infty$, there is a $C>0$ so that
\[ \int_{B(x,r)}|f-f_r|^{2} d\mu \leq\int_{B(x,r)}|f-f_{B_r ,dx}|^{2} d\mu \leq Cr^2 \int_{B(x,r)}|\nabla f|^2 d\mu, \quad\forall f\in C^{\infty}(B(x,r)),\]
for any remote ball $B(x,r)$, where $f_r$ is the average with respect to the $g$ volume $d\mu_g$.

Thus by Lemma~\ref{lem:rem-anch} it remains to prove the Poincar\'{e} inequality for anchored balls.
As in the proof of Lemma~\ref{lem:v-d} we may assume the end $X\setminus\ol{D}_{r_0}$ of $(X,g)$ is conical.
Fix $R>r_0$ and $\kappa>1$.  And choose $\delta>0$ so that $s=\delta R<\frac{1}{3}\inj(g)$.  Let $A=A(R,\kappa R)$ and
$A_\delta$ the $\delta R$-neighborhood of $A$.  Let $\{x_i \}_{i\in I}$ be a maximal subset of $A$ so that the
distance between any two elements is at least $s$.  Let $V_i =B(x_i ,s)$ and $V_i^* =V_i^{\#} =B(x_i ,3s)$.
Then $(V_i,V_i^*,V_i^{\#})_{i\in I}$ is a finite good cover of $A$ in $A_\delta$.
By uniformly bounding $g$ as above, there is a constant $S_c$ so that $\Lambda(V_i,V_i^*)\leq S_c$ and
$\Lambda(V_i^*,V_i^{\#})\leq S_c$.  In other words, the covering satisfies the continuous $L^2$ Sobolev inequality
of order $\nu=\infty$.  The associated graph $(\mathcal{G},m)$ if finite and connected, thus there is a $S_d >0$ so that
the discrete $L^2$ Sobolev-Neumann inequality of order $\infty$ holds.  Theorem~\ref{thm:cov-Sobolev-Neu}
gives a constant $S>0$ so that $\Lambda(A,A_\delta)\leq S$.  By considering the homothetic action
$\psi_a :A(R,\kappa R)_\delta \rightarrow A(R',\kappa R')_\delta$ with $R'=aR$ we see there is a constant $C>0$
independent of $R$ so that
\begin{equation}\label{eq:cont-Poin}
\Lambda(A,A_\delta)\leq CR^2, \quad\forall R>r_0.
\end{equation}

Let $R>r_0$ and $\kappa>1$ be as above.  Choose $\delta>0$ so that $R-r_0 >\delta R$, and set $R_i =\kappa^i R$.
By increasing $r_0$ if necessary, we may assume there is an $r_1 >0$ so that $D_{R_0}\subset B(o,r_1)\subset D_{R_1}$.
We define a covering with the following sets
\begin{gather}
A_0 =D_{R_0} \text{  and  } A_i =A(R_{i-1},R_i)\text{  for  }i\geq 1,\\
A^*_i =A_{i-1}\cup A_i \cup A_{i+1},\\
A^{\#}_i =A^*_{i-1}\cup A^*_i \cup A^*_{i+1},\\
\end{gather}
where we assume $A_i =\emptyset$ for $i<0$.  Choose the least $\ell\geq 1$ so that $B(o,r)\subset D_{R_{\ell}}$.
Then $\mathcal{A}=\{(A_i,A^*_i,A^{\#}_i )\}_{i=0}^\ell$ is a good covering of $B(o,r)$ in $D_{R_{\ell +2}}\subset B(o,\kappa^4 r)$.
If we set $\ol{r}=R_\ell$, then from (\ref{eq:cont-Poin}) we have $\Lambda(A_i,A^*_i)\leq C\ol{r}^2$ for $i\geq 1$
and $\Lambda(A^*_i,A^{\#}_i)\leq C\ol{r}^2$ for $i\geq 2$.

It remains to prove the discrete Sobolev-Neumann inequality.  We will show that there is a constant $c>0$
so that the spectral gap (\ref{eq:spec-gap}) satisfies $c<\lambda(\mathcal{G},m)$.  And by (\ref{eq:disc-isop-Neu}) and (\ref{eq:spect-Cheeger})
it suffices to show there is a $C>0$ independent of $\ell$ such that
\begin{equation}
m(U)\leq Cm(\partial U),\quad\forall U\subset\mathcal{V}\text{ such that }m(U)\leq\frac{1}{2}m(\mathcal{V}).
\end{equation}
Let $j:=\max\{i\in\mathcal{V}: (i,i+1)\in\partial(U)\}$.  Then
\[ \begin{split}
m([0,j]) &  =\sum_{i=0}^j m(i) \\
          & =\mu(A_0) +\sum_{i=1}^j \mu(A(\kappa^{i-1}R,\kappa^i R))\\
          & =\mu(A_0)+ \mu(A_1)\sum_{i=1}^j \kappa^{n(i-1)}\\
          & =\mu(A_0) +\mu(A_1)\frac{\kappa^{nj} -1}{\kappa^n -1}
\end{split}\]
Set $C=\frac{\mu(A_0)}{\mu(A_1)} +\frac{1}{\kappa^n -1}$.  Since either $U$ or $\mathcal{V}\setminus U$ is
contained in $[0,j]$,
\begin{multline}
m(U)\leq\min\{m(U),m(\mathcal{V}\setminus U)\}\leq m([0,j]) \\
\leq C\mu(A_1)\kappa^{nj}=Cm(j+1)\leq Cm(\partial U).
\end{multline}
By Theorem~\ref{thm:cov-Sobolev-Neu} we have $\Lambda(B(o,r),D_{R_\ell})\leq C\ol{r}^2$, for $r\geq r_1$ where
$C$ is independent of $\ol{r}$.  From (\ref{eq:loc-Poin}) there is a $C_1>0$ so that
$\Lambda(B(o,r),B(o,r))\leq C_1 r^2$ for $r\leq r_1$.
The proof is completed by observing that $d(o,-)\sim\rho$ on $X\setminus D_{R_0}$.
\end{proof}

\begin{remark}
Unlike Theorem~\ref{thm:Sobolev}, Theorem~\ref{thm:Poin-con} does not generalize to $\AC$ manifolds with more than
one end.  For example, it is known that the Poincar\'{e} inequality does not hold on a connected sum of Euclidean spaces
$\R^n \#\R^n$.  See~\cite{GriSal-Cos} for more information on this.
\end{remark}

\subsection{Laplacian on Asymptotically conical manifolds}\label{subsect:Lap}

We will need some properties of the Laplacian on an asymptotically conical manifold.  These will follow from some good
bounds on the Green's function which follow from Theorem~\ref{thm:para-Harn}.  First we state an elementary lemma whose
proof is an easy exercise
.
\begin{lem}\label{lem:lap}
Let $(X,g)$ be $\AC(\delta,k+\alpha), k\geq 1, 0<\alpha<1.$  Suppose $u\in C^2_{\beta}(X)$ and $v\in C^2_{\gamma}(X)$
where $\beta,\gamma\in\R$ with $\beta+\gamma <2-n$.  Then
\begin{equation}
\int_X u\Delta v d\mu =\int_X v\Delta u d\mu.
\end{equation}
If $\rho$ is a radius function of $(X,g)$, then $\Delta(\rho^{2-n}) \in C^{k-1,\alpha}_{\delta-n}(X)$.
And if $\Omega:=\Vol(S,g_S)$, where $S$ is the link in the conical end, then
\begin{equation}
\int_X \Delta(\rho^{2-n}) d\mu =(n-2)\Omega.
\end{equation}
\end{lem}

Recall, the Green's function satisfies
\begin{equation}
\Delta_y G(x,y)=\delta_x(y), \quad\forall x\in X.
\end{equation}
Equivalently $G(x,y)$ satisfies
\begin{equation}
\int_X G(x,y)\Delta f(y)d\mu(y) =f(x),
\end{equation}
and
\begin{equation}
\Delta_x \int_X G(x,y)f(y)d\mu(y)=f(x),
\end{equation}
for any smooth compactly supported function $f$ on $X$.

If $(X,g)$ satisfies Theorem~\ref{thm:para-Harn}, then $(X,g)$ admits a positive symmetric Green's function if and only if
$\int^\infty V(x,\sqrt{t})^{-1}dt <\infty$.  And in this case
\[ G(x,y) =\int_0^\infty h(t,x,y) dt.\]

If $n>2$, then from Corollary~\ref{cor:Sobolev} and integrating (\ref{eq:Gauss}) we have for some $C>0$ depending only on $g$
\begin{equation}\label{eq:Green-bound}
\begin{split}
0< G(x,y) & \leq c\int_{d(x,y)^2}^\infty \frac{dt}{V(x,\sqrt{t})} \\
  &  \leq Cd(x,y)^{2-n},
\end{split}
\end{equation}
for all $x,y\in X$, where the second inequality uses Corollary~\ref{cor:Sobolev}.

Similarly, by integrating (\ref{eq:Gaus-der}) for $x\neq y$ and $d(y,z)\leq d(x,y)/2$ we have
\begin{equation}\label{eq:Green-bound-dif}
\begin{split}
\frac{|G(x,y)-G(x,z)|}{d(y,z)^\epsilon} & \leq c \int_{d(x,y)^2}^\infty \frac{dt}{t^\frac{\epsilon}{2} V(x,\sqrt{t})} \\
            & \leq C d(x,y)^{2-n-\epsilon},
\end{split}
\end{equation}
where again we have used Corollary~\ref{cor:Sobolev} in the second inequality.

If $f\in C^{k,\alpha}_{\beta}(X)$ with $k\geq 0,\alpha\in (0,1)$ and
$\beta<-2$, then (\ref{eq:Green-bound}) and standard regularity arguments show that
$u(x)=\int_X G(x,y)f(y)d\mu(y)$  is locally in $C^{k+2,\alpha}(X)$
and $\Delta u=f$.  We will extend these arguments to prove the following.

\begin{thm}\label{thm:lap-weight}
Suppose $(X,g)$ is $\AC(\delta,\ell+\alpha),\delta<-\epsilon, \ell\geq 0, \alpha\in (0,1)$ of dimension $n>2$.
Let $k\leq\ell$, then we have the following.
\begin{thmlist}
\item  Suppose $-n<\beta<-2$.  There exists a $C>0$ so that for each $f\in C^{k,\alpha}_\beta (X)$ there is a unique
$u\in C^{k+2,\alpha}_{\beta+2}(X)$ with $\Delta u=f$ which satisfies
$\|u\|_{C^{k+2,\alpha}_{\beta+2}}\leq C\|f\|_{C^{k,\alpha}_\beta}$.
\item  Suppose $-n-\epsilon <\beta<-n$.  There exist $C_1 ,C_2 >0$ such that for each $f\in C^{k,\alpha}_\beta (X)$ there is a
unique $u\in C^{k+2,\alpha}_{2-n}(X)$ with $\Delta u=f$.   Furthermore, if we define
\begin{equation}\label{eq:lap-int}
A=\frac{1}{(n-2)\Omega}\int_X f d\mu,
\end{equation}
where $\Omega=\Vol(S)$, then $u=A\rho^{2-n} +v$ with $v\in C^{k+2,\alpha}_{\beta +2}(X)$ satisfying
$|A|\leq C_1 \|f\|_{C^0_\beta}$ and $\|v\|_{C^{k+2,\alpha}_{\beta +2}}\leq C_2 \|f\|_{C^{k,\alpha}_\beta}$.
\end{thmlist}
\end{thm}
\begin{proof}
We define
\begin{equation}\label{eq:lap-sol}
 u(y)=\int_X G(y,x)f(x)d\mu(x),
\end{equation}
and we first prove that $u\in C^0_{\beta +2}$.  We have
$|f(x)|\leq \|f\|_{C^0_{\beta}}\rho(x)^\beta$, and from (\ref{eq:lap-sol}) and (\ref{eq:Green-bound}) we have
\begin{equation}
|u(y)|\leq C \|f\|_{C^0_{\beta}} \int_X d(y,x)^{2-n} \rho(x)^\beta d\mu(x).
\end{equation}
Let $o\in X$ be a fixed point.  We split the integral into three regions
$R_1 =\{x\in X: 4d(o,x)\leq d(o,y)\}$, $R_2 = \{x\in X: \frac{1}{4}d(o,y)< d(o,x)<2d(o,y)\}$,
and $R_3 =\{x\in X: d(o,x)> 2d(o,y)\}$.  Estimating the integral over the three regions
gives the following:
\[\begin{split}
\int_{R_1} d(y,x)^{2-n} \rho(x)^\beta d\mu(x) & \leq
\begin{cases}
 C'\rho(y)^{\beta+2} & \text{if } \beta\in(-n,-2) \\
 C'\rho(y)^{2-n}     & \text{if } \beta <-n
\end{cases},\\
\int_{R_2} d(y,x)^{2-n} \rho(x)^\beta d\mu(x) & \leq\rho(y)^{\beta+2}, \text{ and }\\
\int_{R_3} d(y,x)^{2-n} \rho(x)^\beta d\mu(x) & \leq\rho(y)^{\beta+2}
\end{split}\]
This proves that $u\in C^0_{\beta +2}(X)$ in part (i).

We now consider part (ii).  If $\Delta u=f$ with $u\in C_{\beta+2}^{k+2,\alpha}(X)$ and $\beta\in (-n-\epsilon,-n)$,
then
\[ \int_X f d\mu = \int_X \Delta u d\mu =0 \]
by the arguments in Lemma~\ref{lem:lap}.  Thus for $f\in C_{\beta}^{k,\alpha}(X)$ there exists $u\in C_{\beta+2}^{k+2,\alpha}(X)$
solving $\Delta u=f$ only if $\int_X f d\mu =0$.  So suppose that $\int_X f d\mu =0$.  And
we replace (\ref{eq:lap-sol}) with
\begin{equation}\label{eq:lap-sol-mod}
 u(y)=\int_X [G(y,x)-G(y,o)]f(x)d\mu(x).
\end{equation}
Since $\int_X f d\mu =0$, this integral is equal to (\ref{eq:lap-sol}) and thus solves $\Delta u=f$.
From (\ref{eq:Green-bound-dif}) we have
\begin{equation}
|u(y)|\leq C \|f\|_{C^0_{\beta}} \int_X d(x,y)^{2-n-\epsilon} d(o,x)^\epsilon \rho(x)^\beta d\mu.
\end{equation}
And it is not difficult to show that
\[\int_{R_1} d(x,y)^{2-n-\epsilon} d(o,x)^\epsilon \rho(x)^\beta d\mu(x) \leq
\begin{cases}
 C'\rho(y)^{\beta+2} & \text{if } \beta\in(-n-\epsilon,-2) \\
 C'\rho(y)^{2-n-\epsilon}     & \text{if } \beta \leq -n-\epsilon
\end{cases}\]
This shows that $u\in C^0_{\beta+2}(X)$ if $\beta\in (-n-\epsilon,-n)$.

In both parts (i) and (ii) we have proved that $\|u\|_{C^0_{\beta+2}}\leq C'\|f\|_{C^0_{\beta}}$.
As remarked above, we have that $u$ is locally in $C^{k+2,\alpha}(X)$.  By taking an appropriate covering of $(X,g)$ using
the conical structure and applying the Schauder interior estimates one can show that there is a $C>0$ so that for
$\Delta u =f$ one has
\begin{equation}
\|u\|_{C_{\beta+2}^{k+2,\alpha}} \leq C\left(\|u\|_{C^0_{\beta+2}} +\|f\|_{C^{k,\alpha}_{\beta}} \right).
\end{equation}
Thus there is a $C>0$ so that $\|u\|_{C^{k+2,\alpha}_{\beta+2}}\leq C\|f\|_{C^{k,\alpha}_{\beta}}$.

We complete the proof of part (ii).  Define $A$ by (\ref{eq:lap-int}).  Then by Lemma~\ref{lem:lap} we have
$\int_X [f-A\Delta(\rho^{2-n})] d\mu =0$.  And also by Lemma~\ref{lem:lap} we have
$f-A\Delta(\rho^{2-n})\in C_{\beta}^{k,\alpha}(X)$.  Thus by what we have already proved there is a
$v\in C^{k+2,\alpha}_{\beta+2}(X)$ with $\Delta v= f-\Delta(A\rho^{2-n})$ with
\begin{equation}\label{eq:lap-bound}
\|v\|_{C^{k+2,\alpha}_{\beta+2}} \leq C\left(\|f\|_{C^{k,\alpha}_\beta} +|A|\|\Delta(\rho^{2-n})\|_{C^{k,\alpha}_\beta}\right).
\end{equation}
Note that
\begin{equation}
\begin{split}
|A|\leq \frac{1}{(n-2)\Omega} \int_X |f|\, d\mu & \leq \frac{\|f\|_{C^0_\beta}}{(n-2)\Omega} \int_X \rho^\beta \, d\mu \\
   & \leq C_1 \|f\|_{C^0_\beta},
\end{split}
\end{equation}
where $C_1 =\frac{1}{(n-2)\Omega}\int_X \rho^\beta \, d\mu$ is finite since $\beta\leq -n$.  And this combined with
(\ref{eq:lap-bound}) completes the proof.
\end{proof}

\section{K\"{a}hler case}

Asymptotically conical K\"{a}hler manifolds will now be considered.  We will begin with some definitions and preliminary
results.  In particular, the link $S$ in a K\"{a}hler cone is far from arbitrary.  It is a \emph{Sasaki} manifold which
can be thought of as an odd dimensional analogue of a K\"{a}hler manifold.  We also consider some Hodge theory which will 
be useful late.

\subsection{Background}

\begin{defn}\label{defn:Sasaki}
A $2m-1$-dimensional Riemannian manifold $(S,g)$ is \emph{Sasaki} if the metric cone $(C(S),\ol{g})$,
$C(S)=\R_{>0} \times S,\ \ol{g}=dr^2 +r^2 g,$ is K\"{a}hler.
\end{defn}
\begin{remark}
This is the succinct definition of a Sasaki manifold.  They were originally defined as a manifold carrying a special type of metric
contact structure.  For more on Sasaki manifolds see the monograph~\cite{BG3}.
\end{remark}

It follows from the definition that the Euler vector field $r\partial_r$ acts holomorphically, i.e. $\mathcal{L}_{r\partial_r} J=0$.
It is also not difficult to show that $\xi=Jr\partial_r$ is a Killing vector field which restricts to
$S=\{r=1\}\subset C(S)$.  Thus $\xi +ir\partial_r$ is a holomorphic vector field on $C(S)$.  The restriction of $\xi$ to
$S$ is the \emph{Reeb vector field} of $(S,g)$.  Sasaki manifolds can be distinguished by the action of the Reeb vector field $\xi$.
If $\xi$ generates a free action of $U(1)$ then the Sasaki structure is \emph{regular}.  The Sasaki structure is
\emph{quasi-regular} if the orbits close but there are non-trivial stablizers.  If the orbits do not close, then the Sasaki
structure is \emph{irregular}.

Let $\eta$ be the dual 1-form to $\xi$ with respect to $g$, that is $\eta =\frac{1}{r^2}\xi\contr\ol{g}$.  Then one can check
\begin{equation}\label{eq:contact}
\eta = -J^* \frac{dr}{r} =2d^c \log r,
\end{equation}
where $d^c =\frac{i}{2}(\ol{\partial} -\partial)$.  The restriction of $\eta$ to $S$ is a contact form with Reeb vector field $\xi$.

Since $\mathcal{L}_{r\partial_r} J=0$, the K\"{a}hler form $\omega$ satisfies
\begin{equation}
\begin{split}
2\omega  & = \mathcal{L}_{r\partial_r}\omega \\
         & = d(r\partial_r \contr\omega) \\
         & = d(r^2 \eta) \\
         & = dd^c(r^2), \quad \text{using (\ref{eq:contact})}.
\end{split}
\end{equation}
Thus the K\"{a}hler form $\omega$ on $C(S)$ has potential $\frac{1}{2}r^2$.

We are in particular interested in Ricci-flat K\"{a}hler cones.  The following easily follows from the warped product
structure of $\ol{g}$.
\begin{prop}
Let $(S,g)$ be a $2m-1$-dimensional Sasaki manifold.  Then the following are equivalent.
\begin{thmlist}
\item  $(S,g)$ is Sasaki-Einstein with Einstein constant $2n-2$.
\item  $(C(S),\ol{g})$ is Ricci-flat K\"{a}hler.
\end{thmlist}
\end{prop}

Of course, a necessary condition for $(C(S),\ol{g})$ to be Ricci-flat K\"{a}hler is that $\mathbf{K}_{C(S)}^\ell$
must be trivial for some positive integer $\ell>0$.  But for $C(S)$ to admit a Ricci-flat K\"{a}hler cone metric with the
given Reeb vector field $\xi$ one must require a little more.
\begin{prop}\label{prop:CY-cond}
A necessary condition for $C(S)$ to admit a Ricci-flat K\"{a}hler cone metric with the same complex structure $J$ and
Reeb vector field $\xi$ is that $\mathbf{K}_{C(S)}^\ell$, for some integer $\ell\geq 1$, admits a nowhere vanishing section
$\Omega$ with $\mathcal{L}_\xi \Omega =im\Omega$.
\end{prop}
If the condition in the proposition holds, then $\Omega$ satisfies
\begin{equation}\label{eq:CY-cond}
\frac{i^{m^2}}{2^m}\Omega\wedge\ol{\Omega}=\frac{1}{m!} e^h \omega^m ,
\end{equation}
where $h\in C^\infty(C(S))$ is basic, meaning that $\xi h=r\partial_r h =0$.  Note that the Ricci form is given by
$\Ric(\omega)=dd^c h$ and is zero precisely when $h$ is constant.

\begin{remark}
\textit{A priori} a K\"{a}hler cone $C(S)$ does not contain the vertex.  But it can be proved that $C(S)\cup\{o\}$, with the vertex $o$,
is an affine variety.  See~\cite{vC4} for a proof of the relevant embedding theorem.
\end{remark}

We now define $\AC$ K\"{a}hler manifolds.
\begin{defn}\label{defn:AC-Kaehler}
Let $(C(S),g_0)$ be a K\"{a}hler cone.  Then we say that a K\"{a}hler manifold $(X,g)$ is asymptotically conical of order
$(\delta, k+\alpha)$, asymptotic to $(C(S), g_0)$,
if there is compact subset $K\subset X$, a compact neighborhood $o\in K_0 \subset C(S)$, and a diffeomorphism
$\phi: X\setminus K \rightarrow C(S)\setminus K_0$ so that
\begin{align}
|\phi_* g -g_0 | \in C_{\delta} ^{k,\alpha} &\quad\text{on } C(S)\setminus K_0 \quad\text{and}\label{eq:AC-K-met}\\
|\phi_* J -J_0 | \in C_{\delta} ^{k,\alpha} &\quad\text{on } C(S)\setminus K_0 .\label{eq:AC-K-J}
\end{align}
We will denote this by K\"{a}hler $\AC(\delta, k+\alpha)$. 
\end{defn}

In many cases the end of $(X,g)$ will be holomorphically a cone.  In this case $\phi: X\setminus K \rightarrow C(S)\setminus K_0$
is a biholomorphism.  And one can show that in this case $\phi$ extends to $\phi:X\rightarrow C(S)\cup\{o\}$, which is therefore
a resolution of $C(S)$.

\subsubsection{Hodge theory}

We review some Hodge theory that will be needed.  In particular, a Hodge decomposition will be needed in a proof
of a weighted version of the $\partial\ol{\partial}$-lemma.

We define the space of $L^2$ harmonic forms, where we assume $g$ is quasi-isometric to a $\AC$ metric,
\begin{equation}
L^2 \mathcal{H}^p (X,g):=\{\eta\in L^2(\Lambda^p X) : \Delta\eta =0\}.
\end{equation}
One can show that $L^2 \mathcal{H}^p (X,g)=\{\eta\in L^2(\Lambda^p X) : d\eta=d^* \eta =0\}$.
If $\eta\in L^2 \mathcal{H}^p (X,g)$, then $\eta\in L^2 _k(\Lambda^p X)$ for all $k\geq 0$.  So
$L^2 \mathcal{H}^p (X,g)$ consists of smooth forms.  Furthermore, one can show that it is finite dimensional, and
if $\tilde{g}$ is quasi-isometric to $g$, as in (\ref{eq:quasi-isom}), there is a natural isomorphism
$L^2 \mathcal{H}^p (X,\tilde{g})\cong L^2 \mathcal{H}^p (X,g)$.  See~\cite{Loc} for proofs of these statements.

The $L^2$ harmonic spaces have been computed in our context.  For the following, note that an $\AC$ manifold
$X$ can be compactified with boundary $\partial X =S$.  And $H^*(X,S)$ is isomorphic to the compactly supported
cohomology $H_c^*(X)$.
\begin{thm}[\cite{HauHunMaz}]\label{thm:cohom}
Let $(X,g)$ be a manifold of dimension n which is $\AC$ up to quasi-isometry.  Then we have the natural isomorphisms
\begin{equation}
L^2 \mathcal{H}^p (X,g) \cong\begin{cases}
H^p(X,S), & p<n/2, \\
\im\bigl(H^p(X,S)\rightarrow H^p(X)\bigr), & p=n/2, \\
H^p(X), & p>n/2.
\end{cases}
\end{equation}
\end{thm}

Recall the Kodaira decomposition theorem, which for arbitrary manifolds gives the orthogonal decomposition
\begin{equation}\label{eq:Kodaira}
L^2(\Lambda^p(X,g)) = L^2 \mathcal{H}^p(X,g) \oplus\overline{dC_0^\infty(\Lambda^{p-1})}\oplus \overline{d^* C_0^\infty(\Lambda^{p+1})},
\end{equation}
where the closure in the last two summands is in $L^2$.

We need a more precise decomposition than (\ref{eq:Kodaira}).  Assume from now on that $(X,g)$ is $\AC(\delta,\ell+\alpha)$
with $\delta<0$ and $\ell\geq 2$.   A difficulty in improving (\ref{eq:Kodaira}) is that the operator
\begin{equation}\label{eq:Lap-forms}
\Delta: L^2_{k+2,\beta} (\Lambda^p X)\longrightarrow L^2_{k,\beta-2} (\Lambda^p X)
\end{equation}
is not Fredholm for arbitrary $\beta\in\R$.  The kernel of (\ref{eq:Lap-forms}) is finite dimensional and the closure of the
range has finite codimension.  The difficulty is that the range is not always closed.
It is a result of~\cite{LocMcO} that there is a discrete set $\mathcal{D}_\Delta \subset\R$ so that (\ref{eq:Lap-forms})
is Fredholm precisely when $\beta\in\R\setminus\mathcal{D}_\Delta$.
But one can define a Banach space $\tilde{L}^2_{k+2,\beta} (\Lambda^p X)$  so that
\begin{equation}\label{eq:Lap-forms-mod}
\Delta: \tilde{L}^2_{k+2,\beta} (\Lambda^p X)\longrightarrow L^2_{k,\beta-2} (\Lambda^p X)
\end{equation}
is Fredholm, and the range of (\ref{eq:Lap-forms-mod}) is the closure of the range of (\ref{eq:Lap-forms}).

For each $\tau\geq\beta$ with $\tau\in\R\setminus\mathcal{D}_\Delta$ let $\tilde{B}_\tau$ be the closure of
$L^2_{k+2,\beta} (\Lambda^p X)$ in
\begin{equation}
B_\tau =\{\eta \in L^2_{k+2,\tau} (\Lambda^p X) : \Delta\eta\in L^2_{k,\beta-2} (\Lambda^p X)\}
\end{equation}
with respect to the norm
\begin{equation}
\|\eta\|_{B_\tau} =\|\eta\|_{L^2_{k+2,\tau}} +\|\Delta\eta\|_{L^2_{k,\beta-2}}.
\end{equation}
When $\tilde{B}_\tau$ is equipped with this norm one can show~\cite{Loc} that
\begin{equation}
\Delta: \tilde{B}_\tau \longrightarrow L^2_{k,\beta-2} (\Lambda^p X)
\end{equation}
is Fredholm with range equal to the closure of the range of (\ref{eq:Lap-forms}).  One can also show that
all the $\tilde{B}_\tau$ are isomorphic Banach spaces.  We define $\tilde{L}^2_{k+2,\beta} (\Lambda^p X)$
to be any one of the $\tilde{B}_\tau$.  In particular, we have
$\tilde{L}^2_{k+2,\beta} (\Lambda^p X)=L^2_{k+2,\beta} (\Lambda^p X)$ with equivalent norms if
$\beta\in\R\setminus\mathcal{D}_\Delta$.  And in general
\begin{equation}
\tilde{L}^2_{k+2,\beta} (\Lambda^p X) \subset\underset{\tau>\beta}{\bigcap}L^2_{k+2,\tau} (\Lambda^p X).
\end{equation}

By our conventions we have $L^2_{0,-m}(\Lambda^p X)=L^2(\Lambda^p X)$ with equal norms.  Thus consider
\begin{equation}\label{eq:Lap-forms-L2}
\Delta: \tilde{L}^2_{2,2-m} (\Lambda^p X)\longrightarrow L^2_{0,-m} (\Lambda^p X)=L^2(\Lambda^p X).
\end{equation}
The cokernel of (\ref{eq:Lap-forms-L2}) is $L^2 \mathcal{H}^p (X,g)$, so we have the following decomposition
refining (\ref{eq:Kodaira}).
\begin{thm}\label{thm:Hodge}
Suppose $(X,g)$ is $\AC(\delta,\ell+\alpha)$ with $\delta<0$ and $\ell\geq 2$.  Then we have
\begin{equation}
L^2(\Lambda^p X) =L^2 \mathcal{H}^p (X,g)\oplus dd^*\bigl(\tilde{L}^2_{2,2-m} (\Lambda^p X) \bigr)\oplus d^*d\bigl(\tilde{L}^2_{2,2-m} (\Lambda^p X) \bigr).
\end{equation}
In particular, if $\eta\in L^2(\Lambda^p X)$, then we have the unique $L^2$ decomposition
$\eta=\sigma +d\zeta +d^* \theta$, where $\zeta\in L^2_{1,\delta}(\Lambda^{p-1} X)$ for all $\delta>1-m$ and
$\theta\in L^2_{1,\delta}(\Lambda^{p+1} X)$ for all $\delta>1-m$.
\end{thm}
Of course when $(X,g)$ is K\"{a}hler the decomposition in Theorem~\ref{thm:Hodge} respects the decomposition into types
$\Lambda^p(X) \otimes\C =\oplus_{r+s=p} \Lambda^{r,s}(X)$ as usual because $\Delta=2\Delta_{\ol{\partial}}$.

We now prove a weighted version of the $\partial\ol{\partial}$-lemma.
\begin{prop}\label{prop:dd-bar}
Let $(X,g)$ be K\"{a}hler $\AC(\delta, \ell +\alpha)$ with $\delta<0,\ \ell\geq 2$, and $H^1(S,\R)=0$.
Suppose $\beta\in(-2m,-m)$ and
$\eta\in C_\beta^{k,\alpha}(\Lambda^{1,1}_{\R} X)$, with $0\leq k\leq\ell$, is a closed real $(1,1)$-form with $[\eta]=0$ in
$H^2(X,\R)$.  Then there exists a unique real function $u\in C_{\beta+2}^{k+2,\alpha}(X)$ with $dd^c u=\eta$.
\end{prop}
\begin{proof}
Recall that if $u$ is a smooth function, then
\begin{equation}
-m\, dd^c u\wedge\omega^{m-1} =\Delta u\wedge\omega^{m}.
\end{equation}
Define $f$ by $-m \eta \wedge\omega^{m-1} =f \omega^{m}$.  So $f\in C_{\beta}^{k,\alpha}(X)$.  By Theorem~\ref{thm:lap-weight}
there is a $u\in C_{\beta+2}^{k+2,\alpha}(X)$ with $\Delta u=f$.  Then $\gamma=\eta- dd^c u$ is an exact 2-form
in $C_\beta^{k,\alpha}(\Lambda^{1,1}X)$.  And since $\gamma\wedge\omega^{m-1} =-\frac{1}{m}(f-\Delta u)\omega^m =0$,
one can show that $\gamma$ satisfies
\begin{equation}\label{eq:dd-bar-bil}
\gamma\wedge\gamma\wedge\omega^{m-2} =\frac{-1}{2m(m-1)}|\gamma|^2\omega^m .
\end{equation}

Since $\beta<-m$, we have $\gamma\in L^2(\Lambda^{1,1}X)$.  Therefore $\gamma=\sigma+d\zeta$ according to Theorem~\ref{thm:Hodge},
with $\sigma\in L^2 \mathcal{H}^{1,1} (X,g)$.  Since $H^1(S)=0$, the homomorphism $H^2(X,S)\rightarrow H^2(X)$ is an inclusion.
It follows from Theorem~\ref{thm:cohom} that $L^2 \mathcal{H}^{1,1} (X,g)$ contains no exact forms.
Thus $\gamma=d\zeta$, where $\zeta\in L^2_{1,\delta}(\Lambda^1 X)$ with $\delta>1-m$.
Since $C_0^\infty(X)$ is dense in $L^2_{1,\delta}(\Lambda^1 X)$, we can choose a sequence $\{\zeta_j\}$ converging
to $\zeta$ in $L^2_{1,\delta}(\Lambda^1 X)$.  From (\ref{eq:dd-bar-bil}) we have
\begin{equation}
\begin{split}
0=\int_X d[\zeta_j \wedge\gamma\wedge\omega^{m-2} ] & =\int_X d\zeta_j \wedge\gamma\wedge\omega^{m-2} \\
                                                    & \rightarrow \frac{-1}{2m(m-1)}\int_X |\gamma|^2 \,\omega^m,
                                                    \quad\text{as }j\rightarrow\infty.
\end{split}
\end{equation}
The convergence follows because $d\zeta_j \rightarrow\gamma$ in $L^2_{0,\delta-1}(\Lambda^2 X)$ and we may assume
$\delta>1-m$ is chosen small enough that $\delta-1+\beta <-2m$.  Therefore $\gamma=0$, and $\eta=dd^c u$.
\end{proof}

\subsection{Calabi conjecture}

On an $\AC$ K\"{a}hler manifold $(X,g,J)$ we consider the Monge-Amp\'{e}re equation
\begin{equation}\label{eq:Monge-Amp}
(\omega+dd^c \phi)^m =e^f \omega^m.
\end{equation}
If $\phi$ is a solution to (\ref{eq:Monge-Amp}) and $\omega'=\omega+dd^c \phi$, then
the respective Ricci forms satisfy
\begin{equation}
\Ric(\omega')=\Ric(\omega)-dd^c f.
\end{equation}
Equation (\ref{eq:Monge-Amp}) was solved by S.-T. Yau~\cite{Yau} for a compact K\"{a}hler manifold $(M,g,J)$ under the necessary assumption that
$\int_M (1-e^f)\, d\mu_g =0$.  This solved a conjecture of E. Calabi.

The Calabi conjecture for $\AC$ K\"{a}hler manifolds was solved independently by S. Bando and R. Kobayashi~\cite{BK2} and
G. Tian and S.-T. Yau~\cite{TY2}.  D. Joyce~\cite{Joy1} gave a more exacting proof for the ALE case which gave more precise information
on the solution.  The proof of D. Joyce applies \textit{mutatis mutandis} to this situation.

\begin{thm}[\cite{Joy1}]\label{thm:Calabi}
Suppose $(X,g)$ is asymptotically conical K\"{a}hler of order $(\delta, j+\alpha)$, where $\delta<-\epsilon$, $0<\alpha<1$ and $3\leq j\leq\infty$.  Let $3\leq k\leq j\leq\infty$.
\begin{thmlist}
\item  If $\beta\in (-2m,-2)$, then for each $f\in C_{\beta}^{k,\alpha}(X)$ there is a unique $\phi\in C_{\beta+2}^{k+2,\alpha}(X)$
so that $\omega+dd^c \phi$ is a positive $(1,1)$-form and $(\omega+dd^c \phi)^m =e^f \omega^m$ on $X$.

\item  If $\beta\in (-2m-\epsilon,-2m)$, then for each $f\in C_{\beta}^{k,\alpha}(X)$ there is a unique
$\phi\in C_{2-2m}^{k+2,\alpha}(X)$ so that $\omega+dd^c \phi$ is a positive $(1,1)$-form and
$(\omega+dd^c \phi)^m =e^f \omega^m$ on $X$.  Furthermore, we have $\phi= A\rho^{2-2m} +\psi$ where
$\psi\in C_{\beta+2}^{k+2,\alpha}(X)$ and
\begin{equation}\label{eq:const-A}
A=\frac{1}{(m-1)\Omega}\int_X (1-e^f)\, d\mu,
\end{equation}
 where $\Omega=\Vol(S)$, $S=\{r=1\}\subset C(S)$.
\end{thmlist}
\end{thm}
\begin{remark}
Of course, by the local theory of elliptic operators, whenever $f\in C^\infty(X)$ we have $\phi\in C^\infty(X)$.
The theorem is written as it is to show the precise global regularity that the proof gives.

Part (i) of Theorem~\ref{thm:Calabi} is already known and was essentially proved in~\cite{TY2}.  See~\cite{Got} for a proof
in the context of manifolds with a conical end.  The contribution here is part (ii)
which gives a sharp estimate on solutions for rapidly decaying $f\in C_{\beta}^{k,\alpha}(X),\ \beta<-2m.$
\end{remark}

The proof of Theorem~\ref{thm:Calabi} goes through as in~\cite[\S\S 8.6-8.7]{Joy1}.  The proof is by the continuity method, and
the essential ingredients are some \textit{a priori} estimates on a solution $\phi$ of (\ref{eq:Monge-Amp}).
The Sobolev inequality (\ref{eq:Sobolev}) is used to prove an \textit{a priori} estimate on $\|\phi\|_{C^0}$.
\textit{A priori} estimates on $\|dd^c \phi\|_{C^0}$ and $\|\nabla dd^c \phi\|_{C^0}$  depending only on $\|\phi\|_{C^0}$,
$\|f\|_{C^3}$, and $\|R\|_{C^1}$, where $R$ is the curvature, due to T. Aubin~\cite{Aub1} and S.-T. Yau~\cite{Yau} are applied
as in~\cite{TY2}.   Then Theorem~\ref{thm:lap-weight} is applied as in~\cite{Joy1} to show $\phi$ is in the appropriate
H\"{o}lder space.

\subsection{Ricci-flat metrics}

Our main motivating for proving Theorem~\ref{thm:Calabi} is the following which is an immediate consequence.

\begin{cor}
Suppose $(X,g,J)$ is $\AC(\delta,\ell+\alpha)$ with $3 \leq\ell\leq\infty$.  Let $3\leq k\leq\ell$, and suppose
the Ricci form of $(X,g,J)$ satisfies
\[\Ric(\omega)=dd^c f, \quad f\in C^{k,\alpha}_\beta(X),\quad\text{with }\beta<-2.\]
Then there exists a $\phi\in C^{k+2,\alpha}_{\beta+2}(X)$ so that $\omega' =\omega+dd^c \phi$ is Ricci-flat,
and the corresponding metric $g'$ converges to $g$ in $C^{k,\alpha}_\gamma (X)$ where $\gamma=\max(\beta, -2m)$.
\end{cor}
Merely observe that for a solution to $(\omega+dd^c \phi)^m =e^f \omega^m$ the Ricci forms of the K\"{a}hler metrics
$\omega$ and $\omega' =\omega+dd^c \phi$ satisfy
\begin{equation}
\Ric(\omega')-\Ric(\omega)=-dd^c \log\left(\frac{\omega'^m}{\omega^m} \right) =-dd^c f.
\end{equation}

We have the following uniqueness result.
\begin{thm}\label{thm:unique}
Suppose $(X,g,J)$ is $\AC(\delta,\ell+\alpha),\ 2\leq\ell\leq\infty,$ and Ricci-flat.  Suppose $g'$ is another
Ricci-flat K\"{a}hler metric which converges to $g$ in $C^{0,\alpha}_{\gamma}(X)$ with $\gamma<-m$, i.e.
$|g' -g|\in C^{0,\alpha}_{\gamma}(X)$, and the Ricci forms satisfy $[\omega']=[\omega]$.  Then $g' =g$.
\end{thm}
\begin{proof}
Set $\eta =\omega' -\omega$.  Then $\eta$ is an exact form, and by Proposition~\ref{prop:dd-bar} there is a
$\phi\in C^{4,\alpha}_{\gamma+2}(X)$ with $dd^c \phi =\eta$.  Since $\phi$ solves
$(\omega+dd^c\phi)^m =e^f \omega^m$, the uniqueness part of Theorem~\ref{thm:Calabi} gives $\phi=0$.
\end{proof}

\subsection{Ricci-flat metrics on resolutions}

The main motivation for proving Theorem~\ref{thm:Calabi} is to construct examples of asymptotically conical Ricci-flat
K\"{a}hler manifolds.  In this section we give a proof of the part of Theorem~\ref{thm:main} concerning compactly
supported K\"{a}hler classes.

In order to apply Theorem~\ref{thm:Calabi} one must start with an $\AC$ K\"{a}hler manifold
$(X,g,\omega)$ with $c_1(X)=0$.  Thus we suppose there is a nowhere vanishing holomorphic $n$-form $\Omega$ on $X$.
Recall that the Ricci form of $(X,g,J)$ is
\begin{equation}\label{eq:Ricci-form}
\Ric(\omega)=dd^c \log\left(\frac{\Omega\wedge\ol{\Omega}}{\omega^m}\right).
\end{equation}
If $f=\log\left(\frac{\Omega\wedge\ol{\Omega}}{\omega^m}\right)$, then a solution $(X,g',\omega')$ to Theorem~\ref{thm:Calabi}
has Ricci form $\Ric(\omega')=\Ric(\omega) -dd^c f =0$.
But order to apply Theorem~\ref{thm:Calabi} one must start with an $\AC$ K\"{a}hler manifold
$(X,g,\omega)$ with \emph{Ricci potential} $f\in C^k_{\beta}(X)$ with $\beta< -2$.  In general, it may be difficult to
find a K\"{a}hler metric on $X$ satisfying this.

The case of a quasi-projective variety $X=Y\setminus D$, where
$D$ is a divisor, supporting the anti-canonical divisor $\mathbf{K}^{-1}_Y$, which admits a K\"{a}hler-Einstein metric
was dealt with in~\cite{TY2} and independently in~\cite{BK2}.  A K\"{a}hler metric $\omega$ on $X$ was perturbed to a K\"{a}hler metric
$\omega_0$ whose Ricci potential $f$ satisfies $f\in C^k_{\beta}(X)$.  The author considered~\cite{vC2} and extension of this result to
some cases where $D$ does not admit a K\"{a}hler-Einstein metric.  One essentially needs to start with an $\AC$ K\"{a}hler metric
$(X,g,\omega)$ which approximates a Ricci-flat metric at infinity to high enough order.

We consider the relatively easy case of a crepant resolution $\pi:\hat{X}\rightarrow X=C(S)\cup\{o\}$ of a Ricci-flat K\"{a}hler
cone $C(S)$.  We will obtain Theorem~\ref{thm:main} of the introduction.  In the following $r$ will denote the the radius function 
on the cone $C(S)$.

Recall that a variety $X$ has rational singularities if for some, and it follows any, resolution $\pi:Y\rightarrow X$
$R^j \pi_* \mathcal{O}_Y =0$ for $j>0$.
\begin{prop}\label{prop:rat-sing}
Let $C(S)$ be a K\"{a}hler cone satisfying Proposition~\ref{prop:CY-cond}.  Then $o\in X=C(S)\cup\{o\}$ is a rational singularity.
In particular, if $\pi:\hat{X}\rightarrow X$ is a resolution, then $H^j(\hat{X},\R)=0$ for $j\geq 1$.
\end{prop}
We use the criterion of H. Laufer and D. Burns for the rationality of an isolated singularity $o\in X$.
If $\Omega$ is an holomorphic n-form on a deleted neighborhood of $o\in X$, then $o\in X$ is rational if and only if
\begin{equation}\label{eq:rational}
\int_U \Omega\wedge\ol{\Omega}<\infty,
\end{equation}
where $U$ is a small neighborhood of $o\in X$.  If $\Omega$ satisfies $\mathcal{L}_\xi =im\Omega$, then (\ref{eq:CY-cond}) is
satisfied.  And one easily see that the inequality (\ref{eq:rational}) holds.

\begin{prop}\label{prop:Kahler-form}
Suppose $\omega$ is a K\"{a}hler metric on $\hat{X}$ with $[\omega]\in H^2_c(\hat{X},\R)$.  Then there exists a K\"{a}hler metric
$\omega_0$ on $\hat{X}$ with $[\omega_0]=[\omega]$ and $\ol{\omega} =\pi_* \omega_0$ on $\{x\in\hat{X}: \rho(x)>R\}$ for some $R>0$, where
$\ol{\omega}=C\frac{1}{2}dd^c(r^2),\ C>0,$ is the K\"{a}hler cone metric on $C(S)$, up to homothety.
\end{prop}
\begin{proof}
Let $E_i,\ i=1,\ldots, d,$ be the prime divisors in the exceptional set $E=\pi^{-1}(o)\subset\hat{X}$.  Since
$[\omega]\in H^2_c(\hat{X},\R)$, it is Poincar\'{e} dual to $\sum_{i=1}^d a_i [E_i]\in H_{2m-2}(\hat{X},\R)$, for $a_i \in\R$.
Thus there exists a closed compactly supported real $(1,1)$-form $\theta$ with $[\theta]=[\omega]$.
Let $\eta=\omega-\theta$.  Then $\eta$ is an exact real $(1,1)$-form on $\hat{X}$.
There exists an $\alpha\in\Omega^1$ with $d\alpha =\eta$.  We have $\alpha=\alpha^{1,0} +\alpha^{0,1}$ where
$\alpha^{0,1}=\overline{\alpha^{1,0}}$.  Then $\ol{\partial}\alpha^{0,1} =0$, and by Proposition~\ref{prop:rat-sing} there exists
a $u\in C^\infty (X,\C)$ with $\ol{\partial}u=\alpha^{0,1}$.  Define $v=\frac{i}{2}(\ol{u}-u)$.  Then
\[\begin{split}
dd^c v =i\partial\ol{\partial}v & =\frac{1}{2}\left(\partial\ol{\partial}v -\partial\ol{\partial}\ol{v}\right) \\
                                & =\frac{1}{2}\left(\partial\ol{\partial}v +\ol{\partial}\partial \ol{v}\right) \\
                                & =\frac{1}{2}\left(\partial\alpha^{0,1} +\ol{\partial}\alpha^{1,0}\right)\\
                                & =\eta.
\end{split}\]

We may assume the radius function $\rho$ on $\hat{X}$ is chosen so that $\rho(x)=\pi^* r(x)$, for $\rho(x)>2$, and $dd^c(\rho^2)\geq 0$.
Let $\mu :\R\rightarrow [0,1]$ be a smooth function with $\mu(t)=0$ for $t>1$ and $\mu(t)=0$ for $t<0$.
Define $\omega_0 =\theta +C\frac{1}{2}dd^c(\rho^2) +dd^c [\mu(\rho -R)v]$.  Choose $R$ large enough that the support of
$\theta$ is contained in $\{\rho <R\}\subset\hat{X}$.  Then for $C>0$
chosen sufficiently large $\omega_0$ is a K\"{a}hler form with the required properties.
\end{proof}

Now suppose $X=C(S)\cup\{o\}$ be a Ricci-flat K\"{a}hler cone.  And let $\pi:\hat{X}\rightarrow X$ be a crepant resolution.
Recall, this means that $\pi^* \mathbf{K}_X =\mathbf{K}_{\hat{X}}$.  Thus $\mathbf{K}_{\hat{X}}$ is trivial.
Let $\Omega$ be the holomorphic $n$-form on $X$ as in Proposition~\ref{prop:CY-cond}.  Then $\pi^* \Omega$ is a nowhere
vanishing holomorphic form on $\hat{X}$, which we again denote by $\Omega$.
If $\omega_0$ is a K\"{a}hler form as in Proposition~\ref{prop:Kahler-form}, then a Ricci-potential of $(\hat{X},g_0,\omega_0)$
is $f=\log\left(\frac{c \Omega\wedge\ol{\Omega}}{\omega_0^m}\right)$ where we choose the constant $c=i^{m^2}\frac{m!}{(2C)^m}$ so
that $f=0$ outside a compact set.  Part (ii) of Theorem~\ref{thm:Calabi} gives a $\phi\in C^\infty_{2-2m}(X)$ of the form
$\phi= A\rho^{2-2m} +\psi$ where $\psi\in C^\infty_{2+\beta}$, where $\beta<-2m$.  And $\omega=\omega_0 +dd^c \phi$
is the K\"{a}hler form of the Ricci-flat K\"{a}hler metric in Theorem~\ref{thm:main}.

Suppose that $\omega'$ is another Ricci-flat K\"{a}hler form with $[\omega']=[\omega]$ and
$|\omega' -\omega|\in C^{0,\alpha}_{\beta}(X)$
with $\beta<-m$.  By Proposition~\ref{prop:dd-bar} there is a smooth $u\in C^{2,\alpha}_{\beta+2}(X)$ with
$\omega' -\omega=dd^c u$.  Then $u$ solves $(\omega +dd^c u)^m =\omega^m$.
The uniqueness result of Theorem~\ref{thm:Calabi} then shows that $u=0$.

The constant $A$ in (\ref{eq:const-A}) turns out to be an invariant of the K\"{a}hler class in $H^2_c (\hat{X},\R)$.
Recall that $\omega_0 =\theta +\frac{1}{2}Cdd^c(\rho^2) +dd^c\left[\mu(\rho-R)v\right]$ where the first and third terms have
compact support.  Expanding and using that the integral of a compactly supported exact form is zero gives
\[\begin{split}
\int_X (1-e^f)\omega_0^m & =\int_X \omega_0^m -c\Omega\wedge\ol{\Omega} \\
                         & =\int_X \theta^m +\bigl(\frac{1}{2}Cdd^c(\rho^2)\bigr)^m -\bigl(\frac{1}{2}Cdd^c(r^2)\bigr)^m \\
                         & =\int_X \theta^m .
\end{split}\]
Therefore, if we consider the K\"{a}hler class $[\omega]\in H_c^*(\hat{X},\R)$, then
\begin{equation}\label{eq:A-cohom}
A=\frac{1}{(m-1)\Omega}\int_X (1-e^f)\, d\mu =\frac{1}{(m-1)m!\Omega}[\omega]^{\cup m},
\end{equation}
where $\Omega =\Vol(S)=\{r=1\}\subset C(S)$.

We also have the following result on the K\"{a}hler potential of Ricci-flat metrics of Theorem~\ref{thm:main}.
\begin{thm}\label{thm:R-f-potent}
Let $\hat{X}$ be a crepant resolution of a Ricci-flat K\"{a}hler cone $X=C(S)\cup\{o\}$.   Then in each K\"{a}hler class
on $\hat{X}$, with $C>0$ as in Proposition~\ref{prop:Kahler-form}, there is a unique Ricci-flat K\"{a}hler metric $g$ with
K\"{a}hler form $\omega$ which satisfies
\begin{equation}\label{eq:R-f-potent}
\pi_*(\omega) =\frac{1}{2}Cdd^c (r^2) +Add^c (r^{2-2m}) +dd^c(\psi),
\end{equation}
on $\{x\in X: r(x)>R\}$.  Here $A$ is given by the K\"{a}hler class $[\omega]$ in (\ref{eq:A-cohom}), and
$\psi\in C^\infty_\gamma(X)$
with $\gamma< 2-2m$.
\end{thm}
\begin{remark}
There remains the question of the optimal $\gamma <2-2m$ giving the decay of $\psi$ in Theorem~\ref{thm:R-f-potent}.
This comes down to finding the largest $\epsilon>0$ in Theorem~\ref{thm:lap-weight}.  In general, we only know $\epsilon>0$.
But in Theorem~\ref{thm:R-f-potent} the $\AC$ K\"{a}hler manifold $(\hat{X}, g_0,\omega_0)$ has ``boundary'' $S$ which is Einstein.
The condition for the Laplacian
\begin{equation}\label{eq:Lp-Fred}
\Delta: L^p_{k+2, \delta}(X)\rightarrow L^p_{k,\delta-2}(X)
\end{equation}
to be Fredholm is well known~\cite{LocMcO}.  There is a family of operators on $S$
\[ I(\Delta,\lambda) =\lambda^2 -(2m-2)\sqrt{-1}\lambda +\Delta_S,\]
for $\lambda\in\C$ where $\Delta_S$ is the Laplacian on $S$.  Then $\Spec(I,\lambda)$ is the set of $\lambda$ for which
\[ I(\Delta,\lambda):L^p_{k+2}(S)\rightarrow L^p_k(S)\]
does not admit a bounded inverse.  In our case
\[ \Spec(I,\lambda)=\{0,\ (2m-2)\sqrt{-1},\ \mu_j^+ \sqrt{-1},\ \mu_j^- \sqrt{-1},\ldots| j=1,2,\ldots\} \]
where $\mu^{\pm}_j \sqrt{-1}$ are the two solutions of $x^2-(2m-2)\sqrt{-1}x+\lambda_j =0$ with $\lambda_j$ the j-th eigenvalue
of $\Delta_S$.  It was shown in~\cite{LocMcO} that if $\im \Spec(I,\lambda)$ denote the imaginary components, then
(\ref{eq:Lp-Fred}) is Fredholm for $-\delta\notin\im \Spec(I,\lambda)$.

If $(S,g_S)$ is Sasaki-Einstein then by Lichnerowicz's Theorem $\lambda_1 \geq 2m-1$ with equality only if $S$ is isometric to
the sphere.  And one can check that $\mu_1^+ \geq 2m-1$ with equality only if $S$ is isometric to a sphere.  And for
$\delta\in (-\mu_1^+ ,2-2m)$ the operator (\ref{eq:Lp-Fred}) is Fredholm with index $-1$.  So in Theorem~\ref{thm:R-f-potent}
one will actually have $\psi\in C^\infty_{1-2m}(X)$. \footnote{Ryushi Goto pointed this out to me.}
\end{remark}

\section{Examples}

We consider some examples of asymptotically conical Ricci-flat K\"{a}hler manifolds given by Theorem~\ref{thm:main}.
These examples are $\AC(2n,\infty)$ Ricci-flat examples, and are either resolutions of toric K\"{a}hler cones or resolutions of hypersurface singularities.  See~\cite{vC3} for many more examples on resolutions of K\"{a}hler cones.  Here we just give enough
examples to give the reader an idea of the scope of examples.

Also in~\cite{vC2} examples are constructed on affine varieties which are of type $\AC(2n,k)$ for large $k>0$.

\subsection{Resolutions of hypersurface singularities}

We describe how examples can be constructed from resolutions of weighted homogeneous hypersurface singularities.
Let $\mathbf{w}=(w_0 ,\ldots, w_m)\in (\Z_+)^{m+1}$ with $\gcd(w_0,\ldots,w_m)=1$.  We have the weighted $\C^*$-action on
$\C^{m+1}$ given by $(z_0,\ldots, z_m)\rightarrow (\lambda^{w_0}z_0,\ldots,\lambda^{w_m}z_m)$ for $\lambda\in\C^*$.  A polynomial
$f\in\C[z_0,\ldots,z_m]$ is \emph{weighted homogeneous} of degree $d\in\Z_+$ if
\begin{equation}
f(\lambda^{w_0}z_0,\ldots,\lambda^{w_m}z_m)=\lambda^d f(z_0,\ldots,z_m).
\end{equation}

There is a \emph{weighted Sasaki structure} on the sphere $S^{2m+1}_{\mathbf{w}}$ for which the Reeb vector field $\xi_{\mathbf{w}}$
generates the $S^1$-action induced by the above weighted action.  See~\cite{vC4} for details.  The cone $C(S^{2m+1}_{\mathbf{w}})$
is biholomorphic to $\C^{m+1} \setminus\{o\}$, but with a much different metric.  If $f$ is a weighted homogeneous polynomial,
then the K\"{a}hler cone structure of $C(S^{2m+1}_{\mathbf{w}})$ restricts to
$X_f =\{z\in\C^{m+1} : f(z)=0\}$.  And similarly the Sasaki structure on $S^{2m+1}_{\mathbf{w}}$ restricts to the link
$S_f := X_f \cap S^{2m+1}$.  This is given in the diagram:

\begin{equation}
\begin{array}{ccc}
X_f    & \hookrightarrow & C(S^{2m+1}_{\mathbf{w}}) \\
\cup &                 & \cup \\
S_f    & \hookrightarrow & S^{2m+1}_{\mathbf{w}} \\
\downarrow &     &  \downarrow \\
Z_f    & \hookrightarrow & \C P(\mathbf{w})
\end{array}
\end{equation}

Here $Z_f$ is a hypersurface in the weighted projective space $\C P(\mathbf{w})$.
\begin{prop}\label{prop:hyp-CY}
The K\"{a}hler cone $C(S_f)$ admits a nowhere vanishing holomorphic $m$-form $\Omega$ satisfying
$\mathcal{L}_{\xi} \Omega =im\Omega$, after possibly rescaling the Reeb vector field $\xi$, if and only if
$d< |\mathbf{w}|=\sum_{j=0}^m w_j$.
\end{prop}
This is precisely the condition that the \emph{orbifold} canonical bundle $\mathbf{K}_{Z_f}$ on $Z_f$ is negative.

With Proposition~\ref{prop:hyp-CY} satisfied, we are interested in transversally deforming the Sasaki structure of $S_f$ to
a Sasaki-Einstein structure.  If $\eta$ is the contact structure of $S_f$ then $\omega^T =\frac{1}{2}d\eta$ is the K\"{a}hler
structure transversal to the foliation generated by the Reeb field $\xi$.  A \emph{transversal deformation} is a new Sasaki structure
with transversal K\"{a}hler form
\begin{equation}
(\omega^T)' =\omega^T +dd^c \phi,
\end{equation}
for some basic $\phi\in C^\infty_B (S)$.  The new contact form is $\eta' =\eta +2d^c \phi$.  And one can show that the K\"{a}hler
structure on the cone becomes $\omega' =\frac{1}{2}dd^c (r')^2$ where $r'=e^\phi r$.  Obtaining a Sasaki-Einstein structure
is equivalent to solving the transversal K\"{a}hler-Einstein condition
\begin{equation}\label{eq:trans-KE}
\Ric((\omega^T)') =2m(\omega^T)'.
\end{equation}

Condition (\ref{eq:CY-cond}) implies that
\begin{equation}
\Ric(\omega^T)-2m\omega^T =dd^c h.
\end{equation}
And solving (\ref{eq:trans-KE}) is equivalent to solving the transversal Monge-Amp\`{e}re equation
\begin{equation}\label{eq:trans-MA}
(\omega^T +dd^c\phi)^{m} =e^{-2m\phi +h}(\omega^T)^{m}.
\end{equation}

See~\cite{BGJ, BG2,BG3} for more on solving (\ref{eq:trans-MA}) to find Sasaki-Einstein metrics.  In particular, if
$f$ is a Brieskorn-Pham polynomial, $f=\sum_{j=0}^m z_j^{a_j}$.   Then \cite[Theorem 34]{BGJ} gives simple numerical conditions
on the $a_j$ for (\ref{eq:trans-MA}) to be solvable.

For example, consider
\begin{equation}
f= z_0^m +z_1^m +\cdots+ z_{m-1}^m +z_m^k.
\end{equation}
Then these conditions are satisfied if $k>m(m-1)$, and $X_k :=\{z\in\C: f(z)=0\}\subset\C^{m+1}$ has a Ricci-flat K\"{a}hler cone structure.

Let $\hat{\C^{m+1}}$ be the blow-up of $o\in\C^{m+1}$.  If $X' \subset\hat{\C^{m+1}}$ is the birational transform and
$E=X' \cap\cps^m \subset\hat{\C^{m+1}}$ is the exceptional divisor, then adjunction gives
\[ \mathbf{K}_{X'} =\pi^* \mathbf{K}_X + (m-\deg f) E.\]
Thus $\pi: X' \rightarrow X_k$ is crepant for $k\geq m$.  It is not difficult to see that $X'$ has one singularity isomorphic to
$X_{k-m}$.  If $k = 0$ or $1 \mod m$, then by repeatedly blowing up $\lfloor\frac{k}{m}\rfloor$ times we get a smooth
crepant resolution $\pi: \hat{X}_k \rightarrow X_k$.  Therefore, if $k>n(n-1)$ and $k = 0$ or $1 \mod m$, then Theorem~\ref{thm:main}
gives a $\lfloor\frac{k}{m}\rfloor$ family of Ricci-flat K\"{a}hler metrics on $\hat{X}_k$ converging to the cone
metric as in (\ref{eq:conv-cpt}).

\subsection{Toric examples}

One easy way to construct examples of Ricci-flat metrics on resolutions is to consider resolutions of toric K\"{a}hler cones.

\begin{defn}
 A K\"{a}hler cone $(C(S),\ol{g}),\ \dim_{\C} C(S) =m$, is toric if it admits an effective isometric action of the torus $T=T^{m}$ which preserves the Euler vector field $r\partial_r$.
\end{defn}
The associated Sasaki manifold $(S,g)$ is said to be toric.  Let $\mathfrak{t}$ be the Lie algebra of $T^m$.  It follows
that the Reeb vector field $\xi\in\mathfrak{t}$.

Since $T^m$ preserves the K\"{a}hler form $\omega=\frac{1}{2}d(r^2 \eta)$ and further preserves $r^2 \eta$, there is a
\emph{moment map}
\begin{equation}\label{eq:moment-map}
\begin{gathered}
 \mu: C(S) \longrightarrow \mathfrak{t}^* \\
 \langle \mu(x),X\rangle = \frac{1}{2}r^2\eta(X_S (x)),
\end{gathered}
\end{equation}
where $X_S$ denotes the vector field on $C(S)$ induced by $X\in\mathfrak{t}$.  We have the
moment cone defined by
\begin{equation}
 \mathcal{C} :=\mu(C(S)) \cup \{0\},
\end{equation}
which from~\cite{Ler} is a strictly convex rational polyhedral cone.  Recall that this means that there are vectors
$u_i,i=1,\ldots,d$ in the integral lattice $\Z_T =\ker\{\exp(2\pi i\cdot):\mathfrak{t}\rightarrow T\}$ such that
\begin{equation}\label{eq:moment-cone}
 \mathcal{C} =\bigcap_{j=1}^{d} \{y\in\mathfrak{t}^* : \langle u_j,y\rangle\geq 0\}.
\end{equation}

The elements $u_j \in\Z_T,\ j=1,\ldots,d,$ span a cone $\mathcal{C}^*$ in $\mathfrak{t}$ dual to $\mathcal{C}$.
Then $\mathcal{C}^*$ and all of its faces define a fan $\Delta$ characterizing $C(S)\cup\{o\}$ as an algebraic toric variety.
(cf.~\cite{Od})  There is a nowhere vanishing holomorphic m-form satisfying Proposition~\ref{prop:CY-cond} precisely when
there is a $\gamma\in \Hom_{\Z}(\Z_T, \Z)$ with $\gamma(u_j) =1,\ j=1,\ldots,d.$   This is the condition that $C(S)\cup\{o\}$
is Gorenstein.

The result of A. Futaki, H. Ono, and G. Wang on the existence of Sasaki-Einstein metrics on toric Sasaki manifolds makes toric
geometry a propitious source of examples.
\begin{thm}[\cite{FOW,CFO}]\label{thm:FOW}
Let $C(S)\cup\{o\}$ be a Gorenstein toric K\"{a}hler cone with toric Sasaki manifold $S$.
Then we can deform the Sasaki structure by varying the Reeb vector field and then performing
a transverse K\"{a}hler deformation to a Sasaki-Einstein metric.  The Reeb vector field and transverse
K\"{a}hler deformation are unique up to isomorphism.
\end{thm}

Define $H_\gamma =\{\gamma =1\}\subset\mathfrak{t}\cong\R^m$.  The intersection $P_{\Delta} =H_\gamma \cap\mathcal{C}^*$ is
an integral polytope in $H_\gamma \cong\R^{m-1}$.
A toric crepant resolution
\begin{equation}\label{eq:toric-resol}
\pi: X_{\tilde{\Delta}}\rightarrow X_{\Delta}
\end{equation}
is given by a nonsingular subdivision
$\tilde{\Delta}$ of $\Delta$ with every 1-dimensional cone $\tau_i \in\tilde{\Delta}(1), i=1,\ldots,N$ generated
by a primitive vector $u_i :=\tau_i \cap H_\gamma$.  This is equivalent to a basic, lattice triangulation of
$P_{\Delta}$.  \emph{Lattice} means that the vertices of every simplex are lattice points, and \emph{basic}
means that the vertices of every top dimensional simplex generates a basis of $\Z^{n-1}$.
Note that a \emph{maximal} triangulation of $P_{\Delta}$, meaning that the vertices of every simplex are its only
lattice points, always exists.  Every basic lattice triangulation is maximal, but the converse only holds in
dimension 2.

We want K\"{a}hler structures on the resolution $X_{\tilde{\Delta}}$.  This is given by a strictly convex support
function $h\in\SF(\tilde{\Delta},\R)$ on $\tilde{\Delta}$.  This is a real valued function which is piecewise linear on
the cones of $\tilde{\Delta}$.  Convexity means that $h(x+y)\geq h(x)+h(y)$ for $x,y\in|\tilde{\Delta}|$, the support of
$\tilde{\Delta}$.  Let $l_\sigma$ define $h$ on the m-cone $\sigma$.  Strict convexity means that
$\langle l_\sigma ,x\rangle\geq h(x)$, for all $x\in|\Delta|$, with equality only if $x\in\sigma$.

The following is proved by taking a torus Hamiltonian reduction of $\C^N$.  See~\cite{BurGuiLer} and also~\cite{vC3}.
\begin{prop}
For each strictly convex support function $h\in\SF(\tilde{\Delta},\R)$ there is a K\"{a}hler structure $\omega_h$ so that
$(X_{\tilde{\Delta}},\omega_h)$ is a Hamiltonian K\"{a}hler manifold and the image of the moment map is the polyhedral set
\[\mathcal{C}_h :=\bigcap_{j=1}^{N} \{y\in\mathfrak{t}^* : \langle u_j,y\rangle\geq\lambda_j\}.\]
If $h$ satisfies $h(u_j)=0$ for $j=1,\ldots,d$, then $[\omega_h ]\in H^2_c (X_{\tilde{\Delta}},\R)$.
The $u_j \in\inter P_{\Delta}, j=d+1,\ldots,N$, correspond to the prime divisors $D_j$ in $E=\pi^{-1}(o)$.
For each $j=d+1,\ldots,N$, let $c_j \in H^2_c (X_{\tilde{\Delta}},\R)$ be
the Poincar\'{e} dual of $[D_j]$ in $H_{2n-2}(X_{\tilde{\Delta}},\R)$.  Then
\[ [\omega_h] =-2\pi\sum_{j=d+1}^N \lambda_j c_j. \]
\end{prop}

If $[\omega_h] \in H^2_c (X_{\tilde{\Delta}},\R)$, then we can apply Proposition~\ref{prop:Kahler-form} to construct
a K\"{a}hler metric $\omega_0$ with Ricci potential $f=\log\left(\frac{c \Omega\wedge\ol{\Omega}}{\omega_0^m}\right)$.  If
$[\omega_h]$ is not compactly supported, then an initial metric is constructed by R. Goto~\cite[\S 5]{Got} with Ricci
potential $f\in C^\infty_{-4}(X)$.  In both cases the initial K\"{a}hler metrics $\omega_0$ and Ricci potentials $f$ can be
taken $T^m$-invariant.  We get the toric version of Theorem~\ref{thm:main}.
\begin{thm}
Let $\pi:X_{\tilde{\Delta}}\rightarrow C(S)\cup\{o\}$ be a crepant resolution of a toric K\"{a}hler cone.  Then for each
strictly convex $h\in\SF(\tilde{\Delta},R)$, there is a $T^m$-invariant Ricci-flat K\"{a}hler metric $g$ whose K\"{a}hler form
satisfies $[\omega]=[\omega_h]$.  If $[\omega_h]\in H_c^2(X_{\tilde{\Delta}},\R)$, then $g$ is unique and is asymptotic to the cone
metric as in (\ref{eq:conv-cpt}), otherwise $g$ converges as (\ref{eq:conv-noncpt}).
\end{thm}

If $[\omega_h]\in H_c^2(X_{\tilde{\Delta}},\R)$, then from (\ref{eq:A-cohom}) we have
\begin{equation}\label{eq:A-toric}
\begin{split}
A = \frac{1}{(m-1)m!\Omega}[\omega]^{m} & =\frac{-2\pi}{(m-1)m!\Omega}\sum_{j=d+1}^N \lambda_j c_j [\omega]^{m-1}\\
                                        & =\frac{-2\pi}{(m-1)m!\Omega}\sum_{j=d+1}^N \lambda_j\int_{E_j} [\omega_h]^{m-1} <0.
\end{split}
\end{equation}
Note that all the quantities in (\ref{eq:A-toric}) can be computed from $h$ in terms of volumes of various polytopes.

When $\dim_{\C} X_\Delta =3$ there always exists a toric crepant resolution $X_{\tilde{\Delta}}$.  And further, if $X_\Delta$
is not the quadric cone $\{z_0^2 +z_1^2 +z_2^2 +z_3^2 =0\}\subset\C^4$, then it admits a toric crepant resolution
$X_{\tilde{\Delta}}$ with a strictly convex support function $h$ so that $[\omega_h] \in H^2_c (X_{\tilde{\Delta}},\R)$.
See~\cite{vC4} for more details.

\bibliographystyle{plain}
\def\cprime{$'$}

\end{document}